\pdfoutput=1
\RequirePackage{ifpdf}
\ifpdf 
\documentclass[pdftex]{sigma}
\else
\documentclass{sigma}
\fi

\usepackage[all]{xy}

\numberwithin{equation}{section}
\numberwithin{figure}{section}
\numberwithin{table}{section}

\DeclareMathOperator{\Pic}{Pic}

\DeclareMathOperator{\im}{im}
\DeclareMathOperator{\rank}{rank}

\DeclareMathOperator{\Spec}{Spec}
\DeclareMathOperator{\id}{Id}
\DeclareMathOperator{\Id}{Id}
\DeclareMathOperator{\NE}{NE}

\DeclareMathOperator{\Tr}{Tr}

\DeclareMathOperator{\val}{val}

\DeclareMathOperator{\TV}{TV}
\DeclareMathOperator{\sign}{sign}
\DeclareMathOperator{\SL}{SL}
\DeclareMathOperator{\GL}{GL}
\DeclareMathOperator{\PSL}{PSL}

\DeclareMathOperator{\trop}{trop}

\DeclareMathOperator{\Aut}{Aut}

\DeclareMathOperator{\Eff}{Eff}

\DeclareMathOperator{\sat}{sat}

\let\?=\overline
\let\:=\colon
\let\bb=\mathbb

\let\rar=\rightarrow

\let\s=\mathcal

\let\wh=\widehat
\let\wt=\widetilde
\let\mb=\mbox

\newcommand {\kk} {\Bbbk}

 \newtheorem{thm}{Theorem}[section]
 
 \newtheorem{lem}[thm]{Lemma}
 \newtheorem{prop}[thm]{Proposition}
 \newtheorem{conj}[thm]{Conjecture}

\theoremstyle{definition}
 \newtheorem{dfn}[thm]{Definition}
 \newtheorem{dfns}[thm]{Definitions}
 \newtheorem{ntn}[thm]{Notation}
 
 \newtheorem{eg}[thm]{Example}
 
 \newtheorem{rmk}[thm]{Remark}
 
 \newtheorem{obss}[thm]{Observations}

 \newtheorem{constr}[thm]{Construction}

\begin{document}
\allowdisplaybreaks

\newcommand{\arXivNumber}{1407.6241}

\renewcommand{\PaperNumber}{042}

\FirstPageHeading

\ShortArticleName{Classification of Rank 2 Cluster Varieties}

\ArticleName{Classification of Rank 2 Cluster Varieties}

\Author{Travis MANDEL}

\AuthorNameForHeading{T.~Mandel}

\Address{School of Mathematics, University of Edinburgh, Edinburgh EH9 3FD, UK}
\Email{\href{mailto:Travis.Mandel@ed.ac.uk}{Travis.Mandel@ed.ac.uk}}
\URLaddress{\url{https://www.maths.ed.ac.uk/~tmandel/}}

\ArticleDates{Received May 09, 2018, in final form May 15, 2019; Published online May 27, 2019}

\Abstract{We classify rank $2$ cluster varieties (those for which the span of the rows of the exchange matrix is $2$-dimensional) according to the deformation type of a generic fiber~$U$ of their $\s{X}$-spaces, as defined by Fock and Goncharov [\textit{Ann.\ Sci.\ \'{E}c.\ Norm.\ Sup\'{e}r.~(4)} \textbf{42} (2009), 865--930]. Our approach is based on the work of Gross, Hacking, and Keel for cluster varieties and log Calabi--Yau surfaces. Call~$U$ positive if $\dim[\Gamma(U,\s{O}_U)] = \dim(U)$ (which equals~2 in these rank~2 cases). This is the condition for the Gross--Hacking--Keel construction [\textit{Publ.\ Math.\ Inst.\ Hautes \'{E}tudes Sci.}\ \textbf{122}
 (2015), 65--168] to produce an additive basis of theta functions on $\Gamma(U,\s{O}_U)$. We find that~$U$ is positive and either finite-type or non-acyclic (in the usual cluster sense) if and only if the inverse monodromy of the tropicalization~$U^{\trop}$ of~$U$ is one of Kodaira's monodromies. In these cases we prove uniqueness results about the log Calabi--Yau surfaces whose tropicalization is $U^{\trop}$. We also describe the action of the cluster modular group on~$U^{\trop}$ in the positive cases.}

\Keywords{cluster varieties; log Calabi--Yau surfaces; tropicalization; cluster modular group}

\Classification{13F60; 14J32}

\section{Introduction}\label{Introduction}

In \cite{FG1}, Fock and Goncharov define a class of schemes, called cluster varieties, whose rings of global regular functions are upper cluster algebras. In \cite{GHK3}, Gross, Hacking, and Keel describe how to view cluster varieties as certain blowups of toric varieties. We review this description, as well as Gross--Hacking--Keel's construction \cite{GHK1} of the tropicalization of a log Calabi--Yau surface. We then use these ideas to give a classification of rank\footnote{Cluster algebraists often take rank $2$ to mean that the exchange matrix is $2\times 2$. However, we use rank to mean the dimension of the space spanned by the rows or columns of the exchange matrix.} $2$ cluster varieties (those for which the symplectic leaves of the $\s{X}$-space are $2$ dimensional) and to describe their cluster modular groups. This can also be viewed as a classification of log Calabi--Yau surfaces.

By a {\it log Calabi--Yau surface} or a {\it Looijenga interior}, we mean a surface $U$ which can be realized as $Y\setminus D$, where $Y$ is a smooth, projective, rational surface over an algebraically closed field $\kk$ of characteristic~$0$, and the {\it boundary}~$D$ is a choice of snc anti-canonical divisor in $Y$. Furthermore, $D=D_1+\cdots + D_n$ is either a cycle of smooth irreducible rational curves $D_i$ with normal crossings, or if $n=1$, $D$~is an irreducible curve with one node. By a {\it compactification} of~$U$, we mean such a pair $(Y,D)$ (\cite{GHK2} calls these compactifications with ``maximal boundary''). We call $(Y,D)$ a {\it Looijenga pair}, as in~\cite{GHK1}. Toric varieties are the most basic examples, and every~$U$ can be obtained by performing certain blowups on a toric surface, cf.\ Lemma~\ref{ToricModel}.

\subsection{Outline of the paper}

{\bf Cluster varieties:} Section~\ref{cluster} reviews \cite{FG1}'s definition of cluster varieties and summarizes~\cite{GHK3}'s description of cluster varieties as certain blowups of toric varieties (up to codimension~$2$). In particular, we review Section~5 of \cite{GHK3}, which shows that log Calabi--Yau surfaces are roughly the same as fibers of rank~$2$ cluster $\s{X}$-varieties. Our classification of cluster varieties will be up to deformation of these associated log Calabi--Yau surfaces. In Sections~\ref{CMGdef} and~\ref{ClusterComplex}, we review \cite{FG1}'s definitions of the cluster modular group $\Gamma$ and the cluster complex~$\s{C}$. Proposition~\ref{TropClust} gives a~simpler definition of~$\Gamma$ by showing that the triviality of cluster transformations can be checked on~$\s{X}^{\trop}$ rather than needing to examine the full~$\s{A}$ and $\s{X}$-spaces.

{\bf The tropicalization of $\boldsymbol{U}$:} In Section~\ref{tropu}, we review \cite{GHK1}'s construction of the tropicaliza\-tion~$U^{\trop}$ of a log Calabi--Yau surface. $U^{\trop}$ is homeomorphic to $\bb{R}^2$, but it has a natural integral linear structure that captures the intersection data of the boundary divisors. The integer points $U^{\trop}(\bb{Z})\subset U^{\trop}$ generalize the cocharacter lattice $N$ for toric varieties, and~$U^{\trop}$ itself generalizes $N_{\bb{R}}:=N\otimes \bb{R}$.

The integral linear structure is singular at a point $0\in U^{\trop}$, and in Section~\ref{mono} we examine the monodromy around this point. In Section~\ref{lines}, we discuss properties of lines in $U^{\trop}$. For example, the monodromy in $U^{\trop}$ may make it possible for lines to wrap around the origin and self-intersect. Section~\ref{LinearAuts} introduces some automorphisms of $U^{\trop}$ that we will see in Section~\ref{CMG} are induced by the action of $\Gamma$. In Section~\ref{Man}, we review some lemmas from \cite{Man1} which will be useful for the classification in Section~\ref{classification}.

Section~\ref{trop-charge} shows that, although $U^{\trop}$ does not in general determine the deformation type of $U$, it does at least determine the {\it charge} of $U$, which is the number of ``non-toric blowups'' necessary to realize a compactification of $U$ as a blowup of a toric variety.

{\bf Classification:} As in \cite{GHK1, GHK2}, log Calabi--Yau surfaces can be classified based on the intersection matrix $(D_i\cdot D_j)_{ij}$ as negative definite, strictly negative semi-definite, or \textit{positive} (meaning not negative semi-definite). Some equivalent characterizations of these cases appear in \cite[Lem\-ma~6.9]{GHK1}, with additional characterizations scattered throughout the various versions of~\cite{GHK1, GHK2}. We review many of these characterizations in Theorems~\ref{NegD},~\ref{NegSD}, and~\ref{Pos}.

We then turn to the main result of this paper, namely, a refinement of the characterization of positive log Calabi--Yau surfaces. These refined classifications are given in Theorem~\ref{all-wrap} (the positive non-acyclic cases), Theorem~\ref{not-all-wrap} (the acyclic cases), and Theorem~\ref{no-wrap} (the finite-type cases), with Proposition~\ref{SomeWrap} separating out the cases which are acyclic but not finite-type. The refined classification is based on several different properties of these varieties, including (but not limited to):
\begin{itemize}\itemsep=0pt
\item The properties of the quiver associated to the cluster variety~-- e.g., Dynkin (finite-type), acyclic, or non-acyclic.
\item The space of global regular functions on $U$~-- e.g., all constant, or including some, all, or no cluster $\s{X}$-monomials.
\item The geometry of $U^{\trop}$, including the monodromy and properties of lines.
\item The intersection form $Q$ on the lattice $D^{\perp}\subset A_1(Y,\bb{Z})$ of curve classes which do not intersect any component of~$D$.
\item The intersection of the Langlands dual cluster complex (a subset of $\s{X}^{\trop}$) with $U^{\trop}$~-- e.g., some, all, or none of~$U^{\trop}$.
\end{itemize}

For example, we find that $U$ corresponds to an acyclic cluster variety if and only if some straight lines in $U^{\trop}$ do not wrap all the way around the origin. The cases where no lines wrap correspond to finite-type cluster varieties. We show that the inverse monodromies of $U^{\trop}$ in these finite-type cases are Kodaira's monodromy matrices $I_n$, $II$, $III$, and $IV$, from his classification of singular fibers in elliptic surfaces in \cite{Kod} (cf.\ Table~\ref{tab:no wrap} for a summary of these cases). Similarly, the non-acyclic positive cases correspond to Kodaira's matrices $I_n^*$, $II^*$, $III^*$, and~$IV^*$~-- furthermore, the intersection form $Q$ on $D^{\perp}$ here is of type $D_{n+4}$ ($n\geq 0$) or $E_n$, $n=8$, $7$, or $6$, respectively (cf.\ Table~\ref{wraptable}). The deformation types for the Kodaira-monodromy cases {\it are} uniquely determined by $U^{\trop}$, and we describe how to construct each of these cases explicitly.

{\bf Cluster modular groups:} \cite{FG1} defines a certain group~$\Gamma$ of automorphisms of cluster varieties, called the {\it cluster modular group}. In Section~\ref{CMG} we explicitly describe the action of~$\Gamma$ on~$U^{\trop}$ in all the positive cases (cf.\ Table~\ref{ClusterModularTable}). This action is interesting because, in addition to capturing most of the relevant data about $\Gamma$, it preserves the scattering diagram which~\cite{GHK1} and~\cite{GHKK} use to construct canonical theta functions on the mirror. Symmetries of the scattering diagram induced by mutations were previously observed in \cite[Theorem~7]{GP}, although they did not put this in the language of cluster varieties or describe the full groups of automorphisms induced in this way.

We end by applying several of the previous results to prove Theorem~\ref{sing}, which says that if the monodromy of $U^{\trop}$ is any of Kodaira's monodromies, then $U^{\trop}$ uniquely determines~$U$ up to a strong version of deformation equivalence that marks~$U$ by its relationship with~$U^{\trop}$.

\section{Cluster varieties as blowups of toric varieties} \label{cluster}

In \cite{FG1}, Fock and Goncharov construct spaces called cluster varieties by gluing together algebraic tori via certain birational transformations called mutations. \cite{GHK3} interprets these mutations from the viewpoint of birational geometry, and thereby relates the log Calabi--Yau surfaces of \cite{GHK1} to cluster varieties. This section will summarize some of the main ideas from~\cite{GHK3}. We do not assume rank~$2$ in this section unless otherwise stated.

\subsection{Defining cluster varieties}\label{ClusterDefs}
The following construction is due to Fock and Goncharov \cite{FG1}.

\begin{dfn}A {\it seed} is a collection of data
\begin{gather*} S=(N,I,E:=\{e_i\}_{i\in I},F,\langle\cdot,\cdot\rangle,\{d_i\}_{i\in I}),\end{gather*}
where $N$ is a finitely generated free Abelian group, $I$ is a finite index set, $E$ is a basis for $N$ indexed by~$I$, $F$ is a subset of $I$, $\langle \cdot,\cdot\rangle$ is a skew-symmetric $\bb{Q}$-valued bilinear form, and the $d_i$'s are positive rational numbers called {\it multipliers}. We call $e_i$ a {\it frozen} vector if $i \in F$. The {\it rank} of a seed or of a cluster variety will mean the rank of $\langle \cdot, \cdot \rangle$.

We define another bilinear form on $N$ by
\begin{gather*} (e_i,e_j):=\epsilon_{ij}:=d_j\langle e_i,e_j \rangle,\end{gather*} and we require that $\epsilon_{ij}\in \bb{Z}$ for all $i,j\in I$.
 Let $M=N^*$. Define\footnote{Beware that our subscripts for $p_1^*$ and $p_2^*$ do not mean the same thing as for~\cite{GHK3}'s $p_1^*$ and $p_2^*$.}
 \begin{gather*} p_1^*\colon \ N\rar M,\qquad v\mapsto (v,\cdot), \qquad p_2^*\colon \ N\rar M,\qquad v\mapsto (\cdot,v).
 \end{gather*} Let $K_i:= \ker (p_i^*)$, $\?{N_i} := \im(p_i^*)\subseteq M$, $\?{e_i}:=p_1^*(e_i)$, and $v_i:=p_2^*(e_i)$. For each $i\in I$, define a~``modified multiplier'' $d_i'$ by saying that $v_i$ is $d_i'$ times a primitive vector in $M$.
\end{dfn}

\begin{rmk}\label{determined}Given only the matrix $(e_i,e_j)$ and the set $F$, we can recover the rest of the data, up to a rescaling of $\langle \cdot,\cdot\rangle$ and a corresponding rescaling of the $d_i$'s. This rescaling does not affect the constructions below, and it is common take the scaling out of the picture by assuming that the $d_i$'s are relatively prime integers (although we do not make this assumption). Also, notice that $\langle\cdot,\cdot\rangle$ and $\{d_i'\}$ together determine $\{d_i\}$, so when describing a seed we may at times give $\{d'_i\}$ instead of $\{d_i\}$.
\end{rmk}

\begin{obss}\label{Cluster Obs}\quad
\begin{itemize}\itemsep=0pt
\item $K_1$ is also equal to $\ker (v\mapsto \langle v,\cdot\rangle )$, so $\langle \cdot,\cdot\rangle$ induces non-degenerate skew-symmetric form on~$\?{N_1}$. This also means that we could have equivalently defined the rank to be that of~$(\cdot,\cdot)$.
\item Define another skew-symmetric bilinear form on $N$ by $[e_i,e_j]:=d_id_j\langle e_i,e_j \rangle$. Then $K_2=\ker\left(v\mapsto [ \cdot,v ] \right)$, so $[e_i,e_j]$ induces a non-degenerated skew-symmetric form on $\?{N_2}$. We can extend this to $\?{N_2}^{\sat}$ (the saturation in $M$ of $\?{N_2}$), and after possibly rescaling $[\cdot,\cdot]$ (and adjusting the $d_i$'s accordingly) we can identify this with the standard skew-symmetric form on $\?{N_2}^{\sat}$ with the induced orientation. We will denote this form and the induced symplectic form on $\?{N_{2,\bb{R}}}$ by $(\cdot\wedge \cdot)$. Here and in the future, $\bb{R}$ in the subscript means the lattice tensored with $\bb{R}$.
\item We note that the seed obtained from $S$ by replacing $\langle \cdot,\cdot\rangle$ with $[\cdot,\cdot]$ and $d_i$ with $d_i^{-1}$ produces the {\it Langlands dual seed} $S^{\vee}$ described in \cite{FG1}.
 Switching to $S^{\vee}$ has the effect of replacing~$(\cdot,\cdot)$ with its negative transpose, thus switching the roles of (and negating)~$p_1^*$ and~$p_2^*$.
\item Since $(\cdot, e_i) = -d_i\langle e_i,\cdot\rangle$, we see that $\im(p^*_2)$ and $\im(v\mapsto \langle v,\cdot\rangle)$ span the same subspace of~$M_{\bb{R}}$. Since the kernel of $v\mapsto \langle v,\cdot\rangle$ is $K_1$ by the first observation above, we see that there is a canonical isomorphism $\?{N_{2,\bb{R}}}\cong \?{N_{1,\bb{R}}}$. One checks that this is a symplectomorphism with respect to the symplectic forms induced by $[\cdot,\cdot]$ and $\langle \cdot,\cdot\rangle$.
\end{itemize}
\end{obss}

Given a seed $S$ as above and a choice of non-frozen vector $e_j\in E$, we can use a {\it mutation} to define a new seed $\mu_j(S):=(N,I,E'=\{e_i'\}_{i\in I},F,\langle\cdot,\cdot\rangle,\{d_i\})$, where the $(e_i')$'s are defined by
\begin{gather}\label{seedmut}
 e'_{i} = \mu_{j}(e_{i}): = \begin{cases}
 e_{i}+\epsilon_{ij}e_{j} &\mb{if $\epsilon_{ij} >0$}, \\
 -e_{i} &\mb{if $i=j$}, \\
	e_{i} &\mb{otherwise}.
	\end{cases}
\end{gather}
Mutation with respect to frozen vectors is not allowed. Note that although the bases change, the form $\langle \cdot,\cdot\rangle$ does not, so $K_1$ and $\?{N}_1^{\sat}$ are invariant under mutation. The same is true for $K_2$ and $\?{N_2}^{\sat}$, as can similarly be seen using the Langlands dual seed and $[\cdot,\cdot]$~-- one can check that the procedure for obtaining $S^{\vee}$ from $S$ commutes with mutation.

Given a lattice $L$ and some $v\in L^*$, we will denote by $z^v$ the corresponding monomial on $T_L:=L\otimes \kk^*=\Spec \kk[L^*]$ (more precisely, max-$\Spec$ of $\kk[L^*]$). Corresponding to a seed $S$, we can define a so-called seed $\s{X}$-torus $X_S:=T_{M}=\Spec \kk[N]$, and a seed $\s{A}$-torus $A_S:=T_{N}=\Spec \kk[M]$.
 We define {\it cluster monomials} $X_i:=z^{e_i}\in \kk[N]$ and $A_i:=z^{e_i^*}\in \kk[M]$, where $\{ e_i^*\}_{i\in I}$ is the dual basis to $E$.

\begin{rmk}\label{convention-departure}In place of $M$, other authors may use the superlattice $(M)^{\circ}\subset M\otimes \bb{Q}$ spanned over $\bb{Z}$ by vectors $f_i:=d_i^{-1}e_i^*$. One then takes $A_i:=\big(z^{f_i}\big) \in \kk[M^{\circ}]$. This seems to significantly complicate the exposition and the formulas that follow with little or no benefit for us, and so we do not follow this convention.
\end{rmk}

For any $j\in I$, we have a birational morphism $\mu_j^{\s{X}}\colon \s{X}_S\rar \s{X}_{\mu_j(S)}$, called a cluster $\s{X}$-mutation, defined by
\begin{gather*}
	\big(\mu_j^{\s{X}}\big)^* X_i' = 	X_i\big(1+X_j^{\sign(-\epsilon_{ij})}\big)^{-\epsilon_{ij}} \qquad \text{for} \quad i\neq j, \qquad \big(\mu_j^{\s{X}}\big)^* X_j' = X_j^{-1} .
\end{gather*}
Similarly, we can define a cluster $\s{A}$-mutation $\mu_j^{\s{A}}\colon \s{A}_S\rar \s{A}_{\mu_j(S)}$,
\begin{gather*}
	A_{j}\big(\mu_j^{\s{A}}\big)^* A_{j}' =\prod_{i\colon \epsilon_{ji}>0} A_{i}^{\epsilon_{ji}} + \prod_{i\colon \epsilon_{ji}<0} A_{i}^{-\epsilon_{ji}},\qquad \big(\mu_j^{\s{A}}\big)^*	A_{i}'=A_{i} \qquad \text{for} \quad i\neq j.
\end{gather*}

Now, the cluster $\s{X}$-variety $\s{X}$ is defined by using compositions of~$\s{X}$-mutations to glue~$\s{X}_{S'}$ to~$\s{X}_S$ for every seed $S'$ which is related to $S$ by some sequence of mutations. Similarly for the cluster $\s{A}$-variety $\s{A}$, with $\s{A}$-tori and $\s{A}$-mutations. The {\it cluster algebra} is the subalgebra of~$\kk[M]$ generated by the cluster variables~$A_i$ of every seed that we can get to by some sequence of mutations. In this context, the well-known Laurent phenomenon simply says that all the cluster variables are regular functions on~$\s{A}$. The ring of all global regular functions on~$\s{A}$ is called the {\it upper cluster algebra}.

On the other hand, the $X_i$'s do not always extend to global functions on $\s{X}$. When a monomial on a seed torus (i.e., a monomial in the $X_i$'s for a fixed seed) does extend to a global function on~$\s{X}$, we call it a {\it global monomial}, as in~\cite{GHK3}.

\subsubsection{Quivers and seeds}\label{quivers}

We now describe a standard way to represent the data of a seed with the data of a (decorated) quiver. Each seed vector $e_i$ corresponds to a vertex $V_i$ of the quiver. The number of arrows from~$V_i$ to~$V_j$ is equal to $\langle e_i , e_j\rangle$, with a negative sign meaning that the arrows actually go from~$V_j$ to~$V_i$. Each vertex $V_i$ is decorated with the number $d_i$. Furthermore, the vertices corresponding to frozen vectors are boxed. Observe that all the data of the seed can be recovered from the quiver.

Now, a seed is called {\it acyclic} if the corresponding quiver contains no directed paths that do not pass through any frozen (boxed) vertices. A cluster variety is called acyclic if any of the corresponding seeds are acyclic. It is easy to see that a~seed~$S$ is acyclic if and only if there is some closed half-space in $\?{N_2}$ which contains $v_i$ for every $i\in I\setminus F$.

\subsection{The geometric interpretation}\label{geom int}

As in \cite{GHK3}, for a lattice $L$ with dual $L^*$ and with $u \in L$, $\psi \in L^*$, and $\psi(u)=0$, define
\begin{gather*}
	\mu_{u,\psi,L}\colon \ T_L \dashrightarrow T_L, \\
	\mu_{u,\psi,L}^*\big(z^{\varphi}\big) = z^\varphi\big(1+z^{\psi}\big)^{-\varphi(u)} \qquad \text{for} \quad \varphi\in L^*.
\end{gather*}

One can check that the mutations above satisfy
\begin{gather*}
	\big(\mu_j^{\s{X}}\big)^*=\mu_{(\cdot,e_j),e_j,M}^*\colon \ z^{v} \mapsto z^{v}\big(1+z^{e_j}\big)^{-(v,e_j)},\\
	\big(\mu_j^{\s{A}}\big)^*=\mu_{e_j,( e_j,\cdot ),N}^*\colon \ z^{\gamma} \mapsto z^{\gamma}\big(1+z^{( e_j,\cdot)}\big)^{-\gamma(e_j)}.
\end{gather*}

\begin{dfn}\label{coprime}A seed $S$ is called {\it coprime} if $v_i$ is not a positive rational multiple of $v_j$ for any distinct $i,j \in I\setminus F$. $S$ is called {\it totally coprime} if every seed mutation equivalent to $S$ is coprime.
\end{dfn}

The following key lemma, compiled from Section~3 of~\cite{GHK3}, is what leads to the nice geometric interpretations of mutations and cluster varieties.

\begin{lem}[\cite{GHK3}]\label{mugeom}Suppose that $u$ is primitive in a lattice $L$. Let $\Sigma$ be a fan in $L$ with rays corresponding to $u$ and $-u$, and let $\TV(\Sigma)$ be the corresponding toric variety. Denote
\begin{gather*}
 F:=\big\{1+z^{\psi}=0\big\}\subset \TV(\Sigma),
\end{gather*}
and define
\begin{gather*}
H^+:=F\cap D_{u}.
\end{gather*}
Then the result of blowing up $H^+$, followed by blowing down the proper transform of $F$, is a new toric variety $\TV(\Sigma')$. Let $\?{\mu}_{u,\psi,L}\colon \TV(\Sigma)\dashrightarrow \TV(\Sigma')$ be the associated birational map. Then the mutation $\mu_{u,\psi,L}\colon T_L\rar T_L$ is the restriction of $\?{\mu}_{u,\psi,L}$ to the big torus orbits.

In general, since $\mu_{ku,\psi,L}=\mu_{u,\psi,L}^k$ $($the $k$-th power with respect to composition$)$, it follows that~$\mu_{ku,\psi,L}$ can be described by repeating this blowup-blowdown procedure $k$ times.

Furthermore, $\mu_j^{\s{X}}$ preserves the centers $\big\{1+z^{e_i}=0\big\}\cap D_{v_i}$ of the blowups corresponding to~$\mu_i^{\s{X}}$ for each $i \neq j$. If $S$ is totally coprime, then $\mu_j^{\s{A}}$ preserves the centers $\big\{1+z^{\?{e_i}}=0\big\}\cap D_{e_i}$ for the blowups corresponding to $\mu_i^{\s{A}}$ for each $i \neq j$.
\end{lem}

In the setup of the lemma above, recall that the projection $L\rar L/\bb{\bb{Z}}\langle u \rangle$ induces on $TV(\Sigma)$ a $\bb{P}^1$-fibration $\pi_u\colon \TV(\Sigma)\rar D_u$ with $D_u$ and $D_{-u}$ as sections. We find it helpful to think of $F$ as $\pi_u^{-1}(H^+)$, or alternatively, as the fibers of $\pi_u$ which intersect $H^+$, cf.\ Fig.~\ref{mutation-fig}.

\begin{figure}[t]\centering
\resizebox{5 in}{1.25 in}{
\xy
(-60,-30)*{}; (-25,-30)*{} **\dir{-}="Dun1";
(-60,-10)*{}; (-25,-10)*{} **\dir{-}="Dup1";
(-40,-8)*{}; (-40,-32)*{} **\dir{-}="F1";
(-60,-27)*{D_{u}};
(-60,-7)*{D_{-u}};
(-42,-20)*{F};
(-5,10)*{}; (30,10)*{} **\dir{-}="Dun2";
(-5,30)*{}; (30,30)*{} **\dir{-}="Dup2";
(15,31)*{}; (23,18)*{} **\dir{-}="F2";
(15,9)*{}; (23,22)*{} **\dir{-}="E2";
(-5,13)*{D_{u}};
(-5,33)*{D_{-u}};
(22,26)*{\wt{F}};
(22,14)*{\wt{E}};
(50,-30)*{}; (85,-30)*{} **\dir{-}="Dun3";
(50,-15)*{}; (85,0)*{} **\dir{-}="Dup3";
(70,-5)*{}; (70,-31)*{} **\dir{-}="E3";
(50,-28)*{D_{u}};
(50,-10)*{D_{-u}};
(68,-20)*{E};
{\ar (-11,4);(-19,-4)};
{\ar (36,4);(44,-4)};
(-40.1,-30)*{\bullet};
(-36,-32)*{H^+};
(21.7,20)*{\bullet};
(19,20)*{p};
(70,-6.5)*{\bullet};
(67,-4)*{H^-};
\endxy
}
\caption{A mutation involves blowing up a hypertorus $H^+$ in $D_{u}$ (left arrow) and then contracting the proper transform $\wt{F}$ of the fibers $F$ which hit $H_+$ (right arrow), down to a hypertorus $H^-$ in $D_{-u}$. $\wt{E}$~denotes the exceptional divisor, with $E$ being its image after the contraction of $\wt{F}$. The locus $p=\wt{E}\cap\wt{F}$ has codimension $2$ and does not appear in the cluster variety.}\label{mutation-fig}
\end{figure}

We now take a closer look at the case of $\s{X}$-mutations. Let $F:=\{\s{X}_j = -1\}$. Then Lemma~\ref{mugeom} tells us that $\big(\mu_j^{\s{X}}\big)^*$ corresponds to blowing up $H^+:=F \cap D_{v_j}$, followed by blowing down the proper transform of $F$, and repeating for a total of $d_j'$ times (with $F$ being replaced after each blowup-blowdown with the newest exceptional divisor). The new seed torus is only different from the old one in that it is missing the blown-down fibers of the initial $\bb{P}^1$-fibration, but has gained the exceptional divisor from the final blowup (except for the lower-dimensional set of points where this exceptional divisor intersects a blown-down fiber, represented by $p$ in Fig.~\ref{mutation-fig}).

Since the centers of the blowups corresponding to the other mutations have not changed, this shows that the cluster $\s{X}$-variety can be constructed, up to codimension $2$, as follows: For any seed $S$, take a fan in $M$ with rays generated by $\pm v_i$ for each $i$, and consider the corresponding toric variety. For each $i\in I\setminus F$, blow up the hypertorus $\{X_i=-1\}\cap D_{(\cdot,e_{i})}$ $d_i'$ times, and then remove the first $(d_i'-1)$ exceptional divisors. Then up to codimension $2$, the cluster $\s{X}$ variety is the complement of the proper transform of the toric boundary. We denote this complement by $\s{X}^{\star}$. We use $\s{A}^{\star}$ to denote the analogously constructed version of~$\s{A}$.

\begin{rmk}\label{CodimensionOne}In this construction of $\s{X}$, the centers for the hypertori we blow up may intersect if $(\cdot,e_i)=(\cdot,e_j)$ for some $i\neq j$, so some care must be taken regarding the ordering of the blowups. When we write $\s{X}^{\star}$, we implicitly assume that we have fixed some ordering of the blowups, and similarly for $\s{A}^{\star}$. Fortunately, this issue only matters in codimension at least $2$ (cf.~\cite{GHK3} for more details). However, when we consider fibers of $\s{X}$ under the map $\lambda$ introduced below, it is possible that some special fibers will have discrepancies in codimension $1$. As we will see below, $\s{A}$~is a~torsor over what is perhaps the ``most special'' fiber of $\s{X}$. The failure of mutations to preserve the centers of blowups for non-coprime~$\s{A}$, along with the resulting fact that $\s{A}$ and $\s{A}^{\star}$ may differ in codimension $1$, may be viewed as consequences of such codimension~$1$ issues in this special fiber.
\end{rmk}

\begin{rmk}\label{affinization}We have seen that codimension $2$ issues arise as a result of missing points like $p$ in Fig.~\ref{mutation-fig}, and also as a result of reordering the blowups. There are also missing contractible complete subvarieties~-- the $(d_j'-1)$ exceptional divisors we remove when applying $\big(\mu_j^{\s{X}}\big)^*$. We view these issues as being unimportant since they do not affect $\Gamma(\s{X},\s{O}_X)$. When we want to stress that we are only interested in $\s{X}$ or its fibers up to these issues, we will say ``up to irrelevant loci.''
\end{rmk}

\subsection{The cluster exact sequence}\label{cluster exact section}
Observe that for each seed $S$, there is a not necessarily exact\footnote{$\im(M)$ might not be saturated in $K_1^*$, resulting in torsion elements in the quotient.} sequence
\begin{gather*}
	0 \rar K_2 \rar N \stackrel{p_2^*}{\rar} M \rar K_1^* \rar 0.
\end{gather*}
Here, $M\rar K_1^*$ is the map dual to the inclusion $K_1\hookrightarrow N$. Tensoring with $\kk^*$ yields an exact sequence, and one can check (cf.\ Lemma~2.10 of~\cite{FG1}) that this sequence commutes with mutation. Thus, one obtains the exact sequence
\begin{gather*}
	1\rar T_{K_2} \rar \s{A} \stackrel{p_2}{\rar} \s{X} \stackrel{\lambda}{\rar} T_{K_1^*} \rar 1.
\end{gather*}

Let $\s{U}:=p_2(\s{A})=\s{X}_e:=\lambda^{-1}(e)\subset \s{X}$. The sequence $1\rar T_{K_2} \rar \s{A} \rar \s{U} \rar 1$, along with the partially compactified version in \cite[Section~2]{ManCox}, should be viewed as a generalization of the construction of toric varieties as quotients, with $\s{U}$ being the generalization of the toric variety.
 In this paper, we are particularly interested in the fibers of $\lambda$, but cf.\ Remark \ref{Cox} for more on how these relate to $\s{A}$.

\subsection{Looijenga interiors}\label{Looijenga seeds}

Section~5 of \cite{GHK3} shows that Looijenga interiors (i.e., log Calabi--Yau surfaces), as defined in Section~\ref{Introduction}, are exactly the surfaces (up to irrelevant loci, cf.\ Remark~\ref{affinization}) which arise as fibers of~$\lambda|_{\s{X}^{\star}}$ for rank $2$ cluster varieties. We explain this now.

\begin{dfns}
For a Looijenga pair $(Y,D)$ as in Section~\ref{Introduction}, we define a {\it toric blowup} to be a Looijenga pair $(\wt{Y},\wt{D})$ together with a birational map $\wt{Y}\rar Y$ which is a blowup at a nodal point of the boundary $D$, such that $\wt{D}$ is the preimage of $D$. Note that taking a toric blowup does not change the interior $U = Y\setminus D = \wt{Y}\setminus \wt{D}$. We also use the term toric blowup to refer to finite sequences of such blowups.

By a {\it non-toric blowup} $\big(\wt{Y},\wt{D}\big)\rar (Y,D)$, we will always mean a blowup $\wt{Y}\rar Y$ at a non-nodal point of the boundary $D$ such that $\wt{D}$ is the proper transform of $D$. Let $\big(\?{Y},\?{D}\big)$ be a~Looijenga pair where~$\?{Y}$ is a toric variety and~$\?{D}$ is the toric boundary. We say that a birational map $Y \rar \?{Y}$ is a {\it toric model} of $(Y,D)$ (or of $U$) if it is a finite sequence of non-toric blowups.

We say two Looijenga interiors $U_1$ and $U_2$ are deformation equivalent, or of the same \textit{deformation type}, if they admit deformation equivalent compactifications with the same boundary, i.e., if there is a family $(\s{Y},\s{D}) \rar S$ with~$S$ connected, with~$\s{D}\rar S$ the trivial family with fibers~$D$, and with compactificaitons of $U_1$ and $U_2$ appearing as fibers.
\end{dfns}

\begin{lem}[{\cite[Proposition~1.19]{GHK1}}]\label{ToricModel} Every Looijenga pair has a toric blowup which admits a~toric model.
\end{lem}

According to \cite{GHK2}, all deformations of $U$ come from sliding the non-toric blowup points along the divisors $\?{D}_i\subset D$ without ever moving them to the nodes of $D$. We call $U$ {\it positive} if some deformation of $U$ is affine. This is equivalent to saying that $D$ supports an effective $D$-ample divisor, meaning a divisor whose intersection with each component of~$D$ is positive. We will always take the term $D$-ample to imply effective. See Section~\ref{PositiveCase} for other equivalent characterizations of~$U$ being positive.

To see that Looijenga interiors are the same as fibers of $\lambda|_{\s{X}^{\star}}$ for rank $2$ cluster varieties, up to irrelevant loci, we will need the following lemma from~\cite{GHK3}.

\begin{lem}[{\cite[Lemma 5.1]{GHK3}}]\label{Fibers}
Let $H_+$ be the intersection of the zero set of $1+z^{e_i}$ with $D_{v_i}$. Let $t\in T_{K_1^*}$. Then $H_+\cap \lambda^{-1}(t)$ consists of $|\?{e_i}|$ points, where $|\?{e_i}|$ is the index of $\?{e_i}:=p_1^*(e_i)$ in~$\?{N_1}$ $($i.e., $\?{e_i}$ is $|\?{e_i}|$ times a primitive vector in $\?{N_1})$.\footnote{If $\kk$ is not algebraically closed, Lemma \ref{Fibers} might not be true, but it at least holds for $\?{e_i}$ primitive in $\?{N_1}$.}
\end{lem}

Now, in light of Lemmas \ref{ToricModel} and~\ref{Fibers} and the description of $\s{X}^{\star}$ in Section~\ref{geom int}, it is clear that for $\langle \cdot,\cdot \rangle$ rank $2$, every fiber of $\lambda|_{\s{X}^{\star}}$ is a Looijenga interior, up to irrelevant loci. For the converse, we use the following:

\begin{constr}\label{LooijengaConstr}Following Construction 5.3 of \cite{GHK3}, let $U$ be a Looijenga interior. Choose a~compactification $(Y,D)$ admitting a toric model $\pi\colon (Y,D)\rar \big(\?{Y},\?{D}\big)$. Let $N_{\?{Y}}$ be the cocha\-rac\-ter lattice of $\?{Y}$. Let $(\cdot \wedge \cdot)\colon N_{\?{Y}}^2\rar \bb{Z}$ denote the standard wedge form.

Suppose that $\pi$ consists of $d_i'$ non-toric blowups at a point $q_i\in \?{D}_{u_i}$, $i=1,\dots,s$, where $\?{D}_{u_i}$ is the divisor corresponding to the ray $\bb{R}_{\geq 0} u_i \subset N_{\?{Y},\bb{R}}$, $u_i\in N_{\?{Y}}$ primitive.
 We can assume that the $q_i$'s are distinct. We extend this to a set $\?{E}:=\{u_1,\dots,u_s,u_{s+1},\dots,u_m\}$ of not necessarily distinct primitive vectors generating $N_{\?{Y}}$, and we choose positive integers $d_{s+1}',\dots,d_m'$.

Now, let $S$ be the seed with $N$ freely generated by a set $E=\{e_1,\dots,e_m\}$, $I=\{1,\dots,m\}$, $F:=\{s+1,\dots,m\}$, $\{d_i'\}$ as above, and $\langle e_i,e_j \rangle := u_i \wedge u_j$.
 Note that we can identify~$\?{N_2}^{\sat}$ with~$N_{\?{Y}}$ via the identification $v_i=d_i'u_i$. Similarly, we can identify $\?{N_1}\cong N/K_1$ with~$N_{\?{Y}}$ via the identification $\langle e_i,\cdot\rangle = u_i$. Thus, each $\?{e_i}$ is primitive in $\?{N_1}$.

 Using $S$ to construct $\s{X}$, the interpretation of $\s{X}$-mutations from Section~\ref{geom int}, together with Lemma~\ref{Fibers}, reveals that~$U$ is deformation equivalent to the general fibers of~$\lambda$, up to irrelevant loci. A bit more work shows that $U$ is in fact isomorphic to some fiber of~$\s{X}^{\star}$, hence isomorphic to the corresponding fiber of~$\s{X}$ up to irrelevant loci.
\end{constr}

This construction shows that:

\begin{thm}\label{Looijenga}Every Looijenga interior can be identified with a fiber of some $\s{X}^{\star}$ associated to a rank $2$ cluster $\s{X}$-variety. Conversely, up to irrelevant loci, the fibers of $\s{X}^{\star}$ for rank $2$ cluster $\s{X}$-varieties are Looijenga interiors, and general fibers of $\s{X}$ are Looijenga interiors up to codimension $2$.
 \end{thm}
For the last statement, we use that $\s{X}\setminus \s{X}^{\star}$ has codimension $2$ and consists of collections of complete interior curves supported in fibers of $\lambda$. Hence, general fibers of $\s{X}$ and $\s{X}^{\star}$ are equal.

\begin{eg}\label{cubic cluster}
Consider the case where $Y$ is a cubic surface, obtained by blowing up $2$ points on each boundary divisor of $\big(\?{Y}\cong \bb{P}^2,\?{D}=D_1+D_2+D_3\big)$. We can take
\begin{gather*}
	\?{E}=\{(1,0),(1,0),(0,1),(0,1),(-1,-1),(-1,-1)\},
\end{gather*}
with each $d_i=d_i'=1$ and $F$ empty. Then the fibers of the resulting $\s{X}$-variety $\s{X}_1$ correspond to the different possible choices of blowup points on the $D_i$'s. The fiber $\s{U}$ is very special, having four $(-2)$-curves. If we instead take $\?{E}=\{(1,0),(0,1),(-1,-1)\}$ with $\langle \cdot,\cdot\rangle$ given by
	$\left(\begin{smallmatrix}
		0 &1 &-1 \\
	 -1 &0 &1 \\
	 1 &-1 &0
		\end{smallmatrix}\right)$,
		and each $d_i=d_i'=2$, then the fibers of the resulting $\s{X}$-variety~$\s{X}_2$ include only the surfaces constructed by blowing up the same point twice on each $D_i$ and then removing the three resulting $(-2)$-curves. $\s{U}$ is the fiber where the blowup points are colinear and so there is one remaining $(-2)$-curve.
		
The deformation type of the fibers of $\s{X}^{\star}$ has only changed by the removal of certain $(-2)$-curves, i.e., by some irrelevant loci. Note that $\s{X}^{\star}_2=\s{X}_2$, and that $\s{X}_2$ can be identified (after filling in the removed $(-2)$-curves) with a subfamily of $\s{X}^{\star}_1$ whose fibers do not agree with those of $\s{X}_1$ in codimension~$1$.

These examples are well-known: the former corresponds to the Teichm\"{u}ller space of the four-punctured sphere, while the latter corresponds to the Teichm\"{u}ller space of the once-punctured torus (cf.\ \cite[Section~2.7]{FG1}).
\end{eg}

Recall the definition of a coprime seed from Definition \ref{coprime}. Note that a seed being coprime means that for each $i\in I\setminus F$, $d_i'$ is the total number of non-toric blowups taken on the divisor corresponding to $v_i$. We now define a notion which in a sense means being as far from coprime as possible (although the two are not mutually exclusive).

\begin{dfn}\label{EquivalentClusters} We say a seed $S$ is {\it maximally factored} if each $d_i'=1$. Two seeds $S_1$ and $S_2$ (along with the associated cluster varieties) will be called {\it fiberwise-equivalent} if the general fibers of the corresponding $\s{X}$-varieties $\s{X}_1$ and $\s{X}_2$ are of the same deformation type, up to irrelevant loci.
\end{dfn}

\begin{eg}The first seed for the cubic surface in Example \ref{cubic cluster} is maximally factored, while the second seed is totally coprime. The two seeds are clearly fiberwise-equivalent since they both correspond to the cubic surface.
\end{eg}

\looseness=-1 Example \ref{cubic cluster} above demonstrates that we can often change the number of vectors in a~seed without changing the fiberwise-equivalence class of the fibers. For example, consider a~seed $\{N=\bb{Z}\langle E \rangle, I,E=\{e_1,\dots,e_m\},F, \langle \cdot, \cdot \rangle, \{d_i\}\}$ with each $d_i=d_i'$ such that each $\?{e_i}$ is primitive\footnote{Every rank $2$ seed is fiberwise-equivalent to one with this primitivity condition because they all have Looijenga pairs as the fibers of their corresponding $\s{X}^{\star}$. However, this condition can easily be avoided.} in $\?{N_1}$. Given a collection of partitions $d_i=d_{i,1}+\dots+d_{i,b_i}$, $d_{i,j}\in \bb{Z}_{\geq 0}$, we can define a~new seed $S'$ as follows: Let $E'\colon \{e_{i,j}\}$, $i=1,\dots,m$, $j=1,\dots,b_i$, and $N':=\bb{Z}\langle E' \rangle$. Define $\langle e_{i_1,j_1},e_{i_2,j_2}\rangle':=\langle e_{i_1},e_{i_2}\rangle$. We say the pair $(i,j)\in F'$ if $i\in F$. Finally, $d_{i,j}$ is as in the partitions. The corresponding space $\s{X}'$ is fiberwise-equivalent to the original $\s{X}$. By this method, we can show that:
\begin{prop}\label{CoprimeEquiv}
Every seed is fiberwise-equivalent to a coprime seed and to a maximally factored seed. Furthermore, by a sequence of mutation equivalences and fiberwise-equivalences, every seed can be related to a totally coprime seed.
\end{prop}
\begin{proof}For the latter statement, if $S$ is not totally coprime, we mutate to a seed $S'$ which is not coprime, then apply the first statement to take a fiberwise-equivalent seed $S''$ which is coprime. We repeat this if $S''$ is not totally coprime. Since $S''$ has lower dimension than $S$, this process terminates.
\end{proof}

\begin{rmk}\label{Cox}According to \cite[Section~4]{GHK3} and \cite{ManCox}, $\Gamma(\s{A}^{\star},\s{O}_{\s{A}^{\star}})$ is the Cox ring for $\s{X}^{\star}_e$, roughly, $\bigoplus_{\s{L}\in \Pic(\s{X}_e^{\star})} \Gamma(\s{X}_e^{\star},\s{L})$. Similarly over other points of $T_{K_1}^*$ besides the identity $e$ (in fact, over generic points we can drop the superscript $\left.^\star\right.$). Here, the ``irrelevant loci'' actually are relevant since they affect the Picard group. Replacing a maximally factored seed $S$ with some fiberwise-equivalent seed $S'$ corresponds to restricting to some sublattice of $\Pic(\s{X}_e^{\star})$, hence, some corresponding subring of $\Gamma(\s{A}^{\star},\s{O}_{\s{A}^{\star}})$. Alternatively, these $\s{A}$-spaces for $S$ and $S'$ are related by a~procedure introduced in \cite{Dup}, now called ``folding'' in the cluster literature.
\end{rmk}

\subsubsection{The canonical intersection form}\label{canonical intersection}

For $S$ a maximally factored rank $2$ seed and $(Y,D)$ a corresponding Looijenga pair, \cite{GHK3} describes a natural way to identify $K_2:=\ker(p_2^*)$ with $D^{\perp}:=\{C\in A_1(Y,\bb{Z})\,|\,C\cdot D_i =0 \,\forall\,i\}$, thus inducing a canonical symmetric bilinear form $Q$ on $K_2$. This identification of $K_2$ with $D^{\perp}$ is as follows: an element $v:=\sum a_i e_i$ of $K_2$ corresponds to a relation $\sum a_{i} v_i = 0$ in $\?{N}_2^{\sat}$, which we recall from Construction~\ref{LooijengaConstr} can be identified with $N_{\?{Y}}$, where $Y\rar \?{Y}$ is a toric model corresponding to~$S$. Standard toric geometry says that this determines a unique curve class $C_v$ in $\pi^{*}[A_1(\bar{Y})]$ such that $C_v\cdot D_i = \sum a_{j}$ for each $i$, where the sum is over all $j$ such that $D_{v_j}=D_i$. So we can define an isomorphism $\iota\colon K_2 \cong D^{\perp}$ by
\begin{gather*}
	v\mapsto C_v - \sum_i a_{i} E_{i},
\end{gather*}
where $E_i$ is the exceptional divisor corresponding to mutating with respect to $e_i$.

Finally, for $u_1,u_2 \in K_2$, define $Q(u_1,u_2) = \iota(u_1) \cdot \iota(u_2)$. We will see in Section~\ref{classification} that $D^{\perp}$ together with this intersection pairing tells us quite a bit about the deformation type of $U$. In particular, \cite{GHK3} tells us that $U$ is positive if and only if $Q$ is negative definite.

Recall that varying the fiber of $\s{X}$ corresponds to changing the choices of non-toric blowup points on $D$. For some choices of blowup points, certain classes $C$ in $D^{\perp}$ may be represented by effective curves. Let $D^{\perp}_{\Eff} \subseteq D^{\perp}$ be the sublattice generated by the curve classes which are represented by an effective curve on some fiber.

\begin{eg}For the seed from Example \ref{cubic cluster}, $K_2$ is generated by $\{e_2-e_1,e_4-e_3,e_6-e_5,e_1+e_3+e_5\}$. The corresponding curves in $D^{\perp}$ are $\{E_1-E_2,E_3-E_4,E_5-E_6,L-E_1-E_3-E_5\}$, where $E_i$ is the exceptional divisor of the blowup corresponding to $e_i$, and $L$ is a generic line in $\?{Y}\cong \bb{P}^2$. Using $E_i\cdot E_j = -\delta_{ij}$, $L\cdot L=1$, and $L\cdot E_i=0$ for each $i$, one easily checks that this lattice has type $D_4$. On the special fiber $\s{U}$, these four curve classes are effective, so $D^{\perp}_{\Eff} = D^{\perp}$.
\end{eg}

\subsection{Tropicalizations of cluster varieties}\label{ClusterTrop}
\cite{FG1} describes {\it tropicalizations} $\s{A}^{\trop}$ and $\s{X}^{\trop}$ of the spaces $\s{A}$ and $\s{X}$, respectively. Given a seed~$S$, $\s{A}^{\trop}$ can be canonically identified as an integral piecewise-linear manifold with $N_{\bb{R},S}$, and the integral points $\s{A}^{\trop}(\bb{Z})$ of the tropicalization are identified with $N_S$. For a different seed~$\mu_j(S)$, the identification is related by the tropicalization of $\mu_j^{\s{A}}$. This turns out to be the integral piecewise-linear function $\?{\mu_j^{\vee}}\colon N_{\bb{R}}\rar N_{\bb{R}}$: that is, the Langlands dual seed mutation, with the overline indicating that $e_j$ is mapped by the same piecewise-linear function as the other vectors, rather than being negated. Similarly for $\s{X}^{\trop}$ and $\s{X}^{\trop}(\bb{Z})$ using $M_{\bb{R},S}$, $M_S$, and the dual seed mutations. We will use the subscript $S$ to indicate that we are equipping the tropical space with the vector space structure corresponding to the seed $S$.

Our interest in this paper is primarily with the fibers $U$ of $\lambda$. $U^{\trop}$ can be canonically identified\footnote{Another perspective which might be worth exploring in the future would be to identify the tropicalizations of different fibers of $\lambda$ with different fibers of $\lambda^*$, with only the fiber over $e$ corresponding to what we call $U^{\trop}$ here.} with $\?{N_2}\otimes \bb{R} = p_2^*(\s{A}^{\trop}) \subset \s{X}^{\trop}$. Here, $U^{\trop}(\bb{Z})$ is identified with $N_2^{\sat}$, as evidenced in Construction~\ref{LooijengaConstr}. We will spend Section~\ref{tropu} analyzing $U^{\trop}$ in the rank $2$ cases. \cite{GHK1} has shown that in these cases, $U^{\trop}$ has a canonical integral linear structure which is closely related to the geometry of the compactifications $(Y,D)$.

\subsection{The cluster modular group}\label{CMGdef}

A {\it seed isomorphism} $h\colon S\rar S'$ is an isomorphism of the underlying lattices which takes (frozen) seed vectors to (frozen) seed vectors (thus inducing a bijection $h\colon I\rar I'$ taking $F$ to $F'$), such that $d_i=d_{h(i)}$ and $\langle e_i,e_j\rangle = \langle h(e_i),h(e_j)\rangle'$. This induces isomorphisms $h\colon \s{X}\rar \s{X}'$ and $h\colon \s{A}\rar \s{A}'$ given by
	$h^{*}X'_{h(i)} = X_{i}$ and $h^{*}A'_{h(i)} = A_{i}$, respectively, as well as an isomorphism from $\s{U}:=p_2(\s{A})\subset \s{X}$ to $\s{U}':=p_2(\s{A}') \subset \s{X}'$. By a \textit{cluster isomorphism}, we mean these induced isomorphisms of the $\s{X}$ and $\s{A}$ spaces. A {\it seed transformation} is a composition of seed mutations and seed isomorphisms, and a {\it cluster transformation} is a composition of cluster mutations and cluster isomorphisms (i.e., the corresponding maps on $\s{A}$ and $\s{X}$). By a {\it seed auto-transformation}, we mean a seed transformation from a seed to itself, and similarly for a \textit{cluster auto-transformation}. A {\it trivial} seed auto-transformation is a seed transformation which acts\footnote{Here, we do not view mutations as acting on $U^{\trop}$. Rather, each mutation-equivalent seed $S$ gives a~piecewise-linear identification of $U^{\trop}$ with a lattice $U_S^{\trop}$, and a seed auto-transformation $S\rar S'$ induces a~map \mbox{$U_S^{\trop}\rar U_{S'}^{\trop}$}, hence a piecewise-linear automorphism of $U^{\trop}$ (in fact, this is a linear automorphism, cf.\ Lemma~\ref{oil}).} trivially on $\s{X}^{\trop}$. Similarly, a \textit{trivial cluster auto-transformation} is a cluster transformation which acts trivially on $\s{A}$ and $\s{X}$.
	
\begin{dfn}[\cite{FG1}]The {\it cluster modular group} $\Gamma$ is the group of {\it cluster} auto-transformations of a base seed $S$ modulo trivial {\it cluster} auto-transformations.
\end{dfn}

We also define an {\it extended cluster modular group} $\wh{\Gamma}$ by allowing seed isomorphisms to reverse the sign of the skew-symmetric form on $N$. For example, for a toric variety with cocharacter lattice $N$, $\Gamma$ can be thought of as the subgroup of $\SL(N)$ which preserves the fan (consisting of rays corresponding to frozen vectors), whereas $\wh{\Gamma}$ can be thought of as the subgroup of $\GL(N)$ preserving the fan. We will analyze the action of $\Gamma$ on $U^{\trop}$ in Section~\ref{CMG}, and we will briefly point out a couple interesting symmetries coming from $\wh{\Gamma}\setminus \Gamma$ (Remark \ref{ExtendedModular}).

\subsection{The cluster complex}\label{ClusterComplex}
A seed $S$ with seed vectors $e_1,\dots,e_n$ determines a cone $C_S \subset \s{X}_S^{\trop}=(\s{X}_{S^{\vee}})^{\trop}:=M_{\bb{R},S}$ given by $e_i \geq 0$ for all $i\in I\setminus F$. The collection of all such cones in $\big(\s{X}^{\vee}\big)^{\trop}$ for every seed mutation equivalent to $S$ forms a simplicial fan called the {\it cluster complex}, denoted by $\s{C}$, cf.\ \cite[Theorem~0.8]{GHKK}. The generators of the rays of this fan are called $g$-vectors. The cones of $\s{C}$ form a~particularly nice piece of the scattering diagram which~\cite{GHKK} uses for constructing canonical theta functions on the mirror $\s{A}^{\vee}$ to $\s{X}$.

Note that the action of $\Gamma$ on $\s{X}^{\vee}$ induces an action on $\big(\s{X}^{\vee}\big)^{\trop}(\bb{Z})$, and this induces an action on the cluster complex. Here, because it is tricky to make sense of what it means for an action on $\big(\s{X}^{\vee}\big)^{\trop}$ to be linear, we view the cluster complex as a collection of tuples of $g$-vectors rather than a collection of linear spaces they span. As we mentioned at the start of Section~\ref{ClusterTrop}, \cite{GHK3} shows that the tropicalization of mutation indeed agrees with the formula for Langlands dual seed mutation, so the action of $h\in \Gamma$ on $\s{X}^{\trop}(\bb{Z})$ is given by the corresponding seed auto-transformation. In particular, if $h$ is trivial, then any cluster auto-transformation representing it corresponds to a trivial seed auto-transformation. The following proposition shows the converse:

\begin{prop}\label{TropClust}
$\Gamma$ acts faithfully on $\big(\s{X}^{\vee}\big)^{\trop}(\bb{Z})$, and may be equivalently defined as the group of seed auto-transformations of a base seed $S$ modulo trivial seed auto-transformations.
\end{prop}
\begin{proof}If $h \in \Gamma$ acts trivially on $\big(\s{X}^{\vee}\big)^{\trop}(\bb{Z})$, then it acts trivially on $\s{C}$. By \cite[Theorem~0.8]{GHKK}, this means that $h$ acts trivially on the set of equivalence classes of seeds, and thus corresponds to a trivial cluster transformation. For the second statement, note that seed transformations and cluster transformations are in bijection by definition, so the only nontrivial part of this statement is that trivial seed auto-transformations correspond bijectively to trivial cluster auto-transformations. We saw one direction of this immediately before the proposition, and the first statement of the proposition is the reverse direction.
\end{proof}

We note that a similar argument shows that $\wh{\Gamma}$ can also be understood in terms of its action on $\big(\s{X}^{\vee}\big)^{\trop}$. One sees that $\wh{\Gamma}$ is the same as the group of cluster auto-transformations considered in \cite{ASS} (cf.\ their Lemma~2.3).

In Section~\ref{CMG} we will describe the action of the cluster modular group on $U^{\trop}$. In many (conjecturally all) cases, every integral linear automorphism of $U^{\trop}$ is induced by an element of the cluster modular group.

\section[$U^{\trop}$ as an integral linear manifold]{$\boldsymbol{U^{\trop}}$ as an integral linear manifold} \label{tropu}
Recall that $U$ denotes a log Calabi--Yau surface. This section examines $U^{\trop}$ with its canonical integral linear structure defined in \cite{GHK1}.

\subsection{Some generalities on integral linear structures} \label{ils}

A manifold $B$ is said to be {\it $($oriented$)$ integral linear} if it admits charts to $\bb{R}^n$ which have transition maps in $\SL_n(\bb{Z})$. We allow $B$ to have a set $O$ of singular points of codimension at least $2$, meaning that these integral linear charts only cover $B':=B\setminus O$. $B'$ has a canonical set of {\it integral points} which come from using the charts to pull back $\bb{Z}^n\subset \bb{R}^n$. Our space of interest, $B=U^{\trop}$, will be homeomorphic to $\bb{R}^2$ and will typically have a singular point at $0$ (which we say is also an integral point).

$B'$ admits a flat affine connection, defined using the charts to pull back the standard flat connection on $\bb{R}^n$. Furthermore, pulling back along these charts give a local system $\Lambda$ of integral tangent vectors on $B'$. We will be interested in the monodromy of $\Lambda$ around $O$.

\subsubsection{Integral linear functions}\label{int lin fun}

By a {\it linear map} $\varphi\colon B_1\rar B_2$ of integral linear manifolds, we mean a continuous map such that for each pair of integral linear charts $\psi_i\colon U_i\rar \bb{R}^n$, $U_i\subset B_i'$ with $\varphi(U_1)\subset U_2$, we have that $\psi_2\circ \varphi\circ \psi_1^{-1}$ is linear in the usual sense. $\varphi$ is {\it integral linear} if it also takes integral points to integral points. By an {\it integral linear function}, we will mean an integral linear map to $\bb{R}$ with its tautological integral linear structure.

We note that to specify an integral linear structure on an integral piecewise linear manifold (i.e., a manifold where transition functions are integral piecewise linear), it suffices to identify which piecewise linear functions are actually linear. These functions can then be used to construct charts. It therefore also suffices (in dimension $2$) to specify which piecewise-straight lines are straight, since (piecewise-)straight lines form the fibers of (piecewise-)linear functions.

\subsection[Constructing $U^{\trop}$]{Constructing $\boldsymbol{U^{\trop}}$} \label{Ut}

\begin{ntn}\label{UtropNotation}\looseness=-1 Given a toric model $(Y,D) \rar (\?{Y},\?{D})$, let $N$ be the cocharacter lattice corresponding to $(\?{Y},\?{D})$ (contrary to Section~\ref{cluster}'s notation), and let $\Sigma\subset N_{\bb{R}}$ be the corresponding fan. $\Sigma$ has cyclically ordered rays $\rho_i$, $i=1,\dots,n$, with primitive generators $v_i$, corresponding to boundary divisors $\?{D_i}\subset \?{D}$ and $D_i \subset D$. Assume $N_{\bb{R}}$ is oriented so that $\rho_{i+1}$ is counterclockwise of $\rho_i$. Let $\sigma_{u,v}$ denote the closed cone bounded by two vectors $u,v$, with $u$ being the clockwise-most boundary ray. In particular, if $u$ and $v$ lie on the same ray, we define $\sigma_{u,v}$ to be just that ray. We may use variations of this notation, such as $\sigma_{i,i+1}:=\sigma_{v_i,v_{i+1}}$ and $v_{\rho}$ for the primitive generator of some arbitrary ray $\rho$ with rational slope, but these variations should be clear from context.
\end{ntn}

We now use $(Y,D)$ to define an integral linear manifold $U^{\trop}$. As an integral piecewise-linear manifold, $U^{\trop}$ is the same as $N_{\bb{R}}$, with $0$ being a singular point and $U^{\trop}(\bb{Z}):=N$ being the integral points.
 Note that an integral $\Sigma$-piecewise linear (i.e., bending only on rays of $\Sigma$) function $\varphi$ on $U^{\trop}$ can be identified with a Weil divisor of $Y$ via $W_{\varphi}:=a_1D_1+\dots + a_n D_n$, where $a_i = \varphi(v_i)\in \bb{Z}$. We define the integer linear structure of $U^{\trop}$ by saying that a~func\-tion~$\varphi$ on the interior of $\sigma_{i-1,i}\cup\sigma_{i,i+1}$\footnote{We assume here that there are more than $3$ rays in $\Sigma$, so that $\sigma_{i-1,i}\cup \sigma_{i,i+1}$ is not all of $N_{\bb{R}}$. This assumption can always be achieved by taking toric blowups of $(Y,D)$. Alternatively, it is easy to avoid this assumption, but the notation and exposition becomes more complicated. We will therefore continue to implicitly assume that there are enough rays for whatever we are trying to do.} is linear if it is $\Sigma$-piecewise linear and $W_\varphi \cdot D_{i} = 0$. This last condition is (for $n\geq 2$) equivalent to
\begin{gather*}
	a_{i-1} +D_{i}^2 a_{i} + a_{i+1} = 0.
\end{gather*}
Equivalently (as in \cite{GHK1}), if $\varphi\colon \sigma_{i-1,i}\cup \sigma_{i,i+1}\rar \bb{R}^2$ is a chart, then
\begin{gather} \label{linear}
	\varphi(v_{i-1})+D_{i}^2 \varphi(v_i) + \varphi(v_{i+1}) = 0.
\end{gather}
Note that the linear structure is determined by $(Y,D)$ and does not depend on the choice of toric model. In fact, while the construction generalized to higher-dimensions, a special feature of the two-dimensional situation is that toric blowups and blowdowns do not affect the integral linear structure, so as the notation suggests, $U^{\trop}$ and $U^{\trop}(\bb{Z})$ depend only on the interior~$U$.

\begin{eg}
If $(Y,D)$ is toric, then $U^{\trop}$ is just $N_{\bb{R}}$ with its usual integral linear structure. This follows from the standard fact from toric geometry that
$
	\sum_i (C \cdot D_i) v_i = 0
$
for any curve class $C$. Taking non-toric blowups changes the intersection numbers, resulting in a singularity at the origin.
\end{eg}

\begin{rmk}\label{divisorial} Recall from standard toric geometry that any primitive vector $v\in N$ corresponds to a prime divisor $D_v$ supported on the boundary of some toric blowup of $\big(\?{Y},\?{D}\big)$, and a general vector $kv$ with $k\in \bb{Z}_{\geq 0}$ and $v$ primitive corresponds to the divisor $kD_{v}$. Two divisors on different toric blowups are identified if there is some common toric blowup on which their proper transforms are the same (equivalently, if they correspond to the same valuation on the function field). Since taking proper transforms under the toric model gives a bijection between boundary components of $(Y,D)$ and boundary components of $\big(\?{Y},\?{D}\big)$ (and similarly for the boundary components of toric blowups), we see that points of $U^{\trop}(\bb{Z})$ correspond to multiples of divisors on compactifications of $U$.
\end{rmk}

\subsection[Another construction of $U^{\trop}$]{Another construction of $\boldsymbol{U^{\trop}}$}\label{SeedScatter}
We now give another construction of the canonical integral linear structure, this time more closely related to the cluster picture. Given a seed $S$, consider the non-frozen seed vectors $\{e_i\}_{i\in I\setminus F}$. Recall that $v_i:=p_2^*(e_i) \in U^{\trop}:=p_2^*(\s{A}^{\trop})\subset \s{X}^{\trop}$ (cf.\ Section~\ref{ClusterTrop}). The integral linear structure on $U^{\trop}$ agrees with that of the vector space $U^{\trop}_S$ (with the lattice $\?{N_{2,S}}$ as the integral points) on the complement of the rays $\rho_i:=\bb{R}_{\geq 0} v_i$, $i\in I\setminus F$. By repeatedly mutating, this determines the integral linear structure everywhere.

For yet another perspective, consider a line $L$ in $U^{\trop}_S$ which crosses a ray $\rho_i$ as above. Viewed as a piecewise-straight line in $U^{\trop}$ with its canonical integral linear structure, $L$ will bend away from the origin when it crosses $\rho_i$. Lines $L$ which are straight in $U^{\trop}$ will bend towards the origin in $U^{\trop}_S$ as follows: if $u$ is a tangent vector to $L$ on one side of $\rho_i$ which points towards $\rho_i$, then on the other side, $u-|u\wedge v_i| v_i$ will be a tangent vector pointing away from $\rho_i$.
 Another way to state this perspective is that the ``broken lines'' (as in \cite{GHK1} and \cite{GHKK}) in $U^{\trop}$ which are actually straight with respect to the canonical integral linear structure are exactly those which bend towards the origin as much as possible.

\subsection{The developing map}\label{dev}

We now describe a tool from \cite{GHK1} that is useful for doing explicit computations on $U^{\trop}$. Consider the universal cover $\xi\colon \wt{U}_0^{\trop} \rar U^{\trop}_0:=U^{\trop}\setminus \{0\}$. Note that $\wt{U}_0^{\trop}$ has a canonical integral linear structure pulled back from~$U^{\trop}_0$. The integral points are $\wt{U}_0^{\trop}(\bb{Z}) := \xi^{-1}\big[U^{\trop}_0(\bb{Z})\big]$. Furthermore, a ray $\rho\in U^{\trop}_0$ pulls back to a family of rays $\rho^j$, $j\in \bb{Z}=\pi_1\big(U^{\trop}_0\big)$, projecting to~$\rho$ (we arbitrarily choose a ray in $\wt{U}_0^{\trop}$ to be $\rho_0$ and then assign the other indices so that they increase as we go counterclockwise).

Suppose that $v \in \rho_0$ and $v' \in \rho'_0$ are primitive vectors in $\wt{U}_0^{\trop}$ spanning the integral points of~$\sigma_{v,v'}$.
 Then there is a unique linear map $\delta_{\rho,\rho'}\colon \wt{U}_0^{\trop} \rar \bb{R}^2\setminus\{0\}$ such that $\delta_{\rho,\rho'}(v)=(1,0)$ and $\delta_{\rho,\rho'}(v')=(0,1)$. We call this the {\it developing map} with respect to~$\rho$ and~$\rho'$.
 We will often leave off the subscripts if they are not relevant, or we will write~$\delta_{\rho}$ if only
the image $\rho$ of the first ray is relevant. $\delta$ is an integral linear immersion, and $\delta\big(\wt{U}_0^{\trop}(\bb{Z})\big) \subseteq \bb{Z}^2\setminus \{(0,0)\}$. A superscript $j\in \bb{Z}$ on $\delta$ will indicate that we are considering the $j^{\tt{th}}$ sheet of $\delta$ (e.g., $\delta^j(\rho):=\delta(\rho^j)$ for $\rho\in U^{\trop}_0$).

\begin{eg} \label{cubicdevelop}
Consider the cubic surface (as in Example~\ref{cubic cluster}) constructed by taking two non-toric blowups on each of the three boundary divisors $D_1$, $D_2$, and $D_3$ of $\bb{P}^2$. The intersection matrix $H:=(D_i \cdot D_j)$ is $
H=	\left(\begin{smallmatrix}
		-1 &1 & 1 \\
		1 &-1 & 1 \\
		1 &1 & -1
		\end{smallmatrix}\right)
$
and equation~\eqref{linear} (or the construction from charts) implies that $\delta^0_{\rho_{D_1},\rho_{D_2}}(v_3) = (-1,1)$, and $\delta^j(v) = (-1)^j\delta^0(v)$. See Fig.~\ref{DevelopingMaps}(a).
\end{eg}

\begin{eg} \label{M05}Consider $\big(\?{\s{M}_{0,5}} , D = D_1 + \dots + D_5\big)$ constructed from the toric surface $\big(\bb{P}^2,\?{D}=\?{D_1}+\?{D_2}+\?{D_4}\big)$ by making toric blowups at $D_1 \cap D_4$ and $D_2 \cap D_4$, as well as one non-toric blowup on each of $\?{D_1}$ and $\?{D_2}$. We then have five boundary components, each with self-intersec\-tion~$-1$. A developing map takes the rays of the fan to $(1,0), (0,1), (-1,1), (-1,0)$, and $(0,-1)$, respectively, and then restarts with $(1,-1)$ and $(1,0)$. See Fig.~\ref{DevelopingMaps}(b).
\end{eg}

\begin{figure}\centering \begin{tabular}{c c}
\xy
{\ar (0,0); (20,0)}; (20,-3)*{\rho_1^0};
{\ar (0,0); (0,20)}; (3,20)*{\rho_2^0};
{\ar (0,0); (-20,20)}; (-23,19)*{\rho_3^0};
{\ar (0,0); (-20,0)}; (-20,3)*{\rho_1^1};
{\ar (0,0); (0,-20)}; (-3,-20)*{\rho_2^1};
{\ar (0,0); (20,-20)}; (23,-19)*{\rho_3^1};
\endxy
 \vspace{.05 in} \quad & \quad
\xy
{\ar (0,0); (20,0)}; (19,-3)*{\rho_1^0, \rho_2^1};
{\ar (0,0); (0,20)}; (6,20)*{\rho_2^0,\rho_3^1};
{\ar (0,0); (-20,20)}; (-24,17)*{\rho_3^0,\rho_4^1};
{\ar (0,0); (-20,0)}; (-20,3)*{\rho_4^0,\rho_5^1};
{\ar (0,0); (0,-20)}; (-3,-20)*{\rho_5^0};
{\ar (0,0); (20,-20)}; (23,-19)*{\rho_1^1};
\endxy
 \vspace{.05 in} \\
 (a) & (b) \end{tabular}
\caption{(a): Cubic surface developing map. We let $\rho_i^j$ denote $\delta^j_{\rho_{D_1},\rho_{D_2}}(\rho_{D_i})$. (b): $\?{\s{M}}_{0,5}$ developing map, with $\rho_i^j$ labelled for $j=0,1$. \label{DevelopingMaps}}
\label{wrap}
\end{figure}

\subsection{Monodromy about the origin}\label{mono}

We now consider what happens when we parallel transport a tangent vector $v$ in $T_p {U^{\trop}}$ counterclockwise around the origin. We use the embedding of a cone in the tangent spaces of its points (which are all identified via parallel transport in the cone), and we use the notation $\delta^i := \delta^i_{\rho_{D_1},\rho_{D_2}}$.

\begin{eg} \label{focusfocus}Suppose $Y\rar \?{Y}$ consists of a single non-toric blowup on, say, $D_1$. Then $\delta^0(v_1)=\delta^1(v_1)=(1,0)$. However, $\delta^0(v_2) = (0,1)$ while $\delta^1(v_2) = (1,1)$. We can view parallel transporting counterclockwise around the origin as parallel transporting up one sheet on the developing map, and then the monodromy tells us how to write the transported vector in terms of~$\delta^1(v_1)$ and~$\delta^1(v_2)$. Thus, the monodromy is
\begin{gather*}
	\mu=\left(\begin{matrix}
		1 &1 \\
		0 &1
		\end{matrix}\right)^{-1}
	=	\left(\begin{matrix}
		1 &-1 \\
		0 &1
		\end{matrix}\right).	
\end{gather*}
\end{eg}

More generally, recall that if $v=(a,b)$ in the basis $\{(1,0), (0,1)\}$, then in a basis~$\{u_1,u_2\}$, $v$~is given by $\left(\begin{smallmatrix} u_1 & u_2 \end{smallmatrix} \right)^{-1}\left(\begin{smallmatrix} a \\ b \end{smallmatrix} \right)$. Thus, the monodromy is in general given by
\begin{gather*}
 \mu=\left(\begin{matrix}\delta^1(v_{1}) & \delta^1(v_{2}) \end{matrix} \right)^{-1}
\end{gather*} with respect to the basis and developing map $\big\{\delta^0(v_{1})=(1,0),\delta^0(v_{2})=(0,1)\big\}$. Thus, for any \mbox{$x\in U^{\trop}_0$} and $k\in \bb{Z}$, we have $\mu^{-k}\big(\delta^0(x)\big)=\delta^k(x)$, and so we may view $\mu^{-k}$ as the deck transformation corresponding to $k\in \bb{Z}=\pi_1\big(U^{\trop}_0\big)$ (positive $k$ corresponding to counterclockwise paths), i.e., $\mu^{-k}$ acts on $\wt{U}_0^{\trop}$ by raising points up $k$ sheets. In particular, the monodromy determines $U^{\trop}$ as an integral linear manifold: $U^{\trop}$ is the quotient of $\wt{U}_0^{\trop}$ by the action of~$\mu^{-1}$.

The matrices $\mu$ and $\mu^{-1}$ can always be factored into a product of unipotent matrices as follows: choose a toric model in which $k_i$ non-toric blowups are taken on the divisor $D_{v_i}$, for $v_1,\dots,v_s\in N$ cyclically ordered clockwise. Then we have the factorization
\begin{align}
	\mu^{-1} = \mu_{v_s}^{-k_s} \cdots \mu_{v_1}^{-k_1},
\end{align}
where $\mu_{v_i}^{-k_i}$ is given in an oriented unimodular basis $(v_i,v_i')$ by the matrix $	\left(\begin{smallmatrix}
		1 &k_i \\
		0 &1
		\end{smallmatrix}\right)$.
 More generally, in a basis where $v_i = (a,b)$, the corresponding contribution to $\mu^{-1}$ is
\begin{gather}\label{genmut}
	\mu_{(a,b)}^{-k_i} :
		= \left(\begin{matrix}
		1- k_i a b &k_i a^2 \\
		-k_i b^2 &1+k_i a b
		\end{matrix}\right).
\end{gather}

Now $\mu$ can of course be expressed as $\mu_{v_1}^{k_1} \cdots \mu_{v_s}^{k_s}$. Alternatively (following from the fact that $A\mu_v A^{-1}=\mu_{Av}$), the monodromy matrix is given by the product $\mu=(\mu'_{v_s})^{k_s} \cdots (\mu'_{v_1})^{k_1}$ of matrices of the form
\begin{align}\label{hardmu}
	(\mu'_{v_i})^{k_i}:=\mu_{(a_i,b_i)}^{k_i}
		= \left(\begin{matrix}
		1+ k_i a_i b_i &-k_i a_i^2 \\
		k_i b_i^2 &1-k_i a_i b_i
		\end{matrix}\right),
\end{align}
where $(a_1,b_1):=v_1$, and for $i>1$, $(a_i,b_i):=(\mu'_{v_{i-1}})^{k_{i-1}}\cdots (\mu'_{v_1})^{k_1} v_i$. This can be interpreted by saying that before we can apply the monodromy contribution corresponding to $v_i$, we have to let the modifications we have made so far act on $v_i$.

\begin{rmk}\label{factor}We note that we can view these factorizations of $\mu$ as corresponding to factorizations of the singular point into several focus-focus singularities (i.e., singularities with unipotent monodromy) which are contained on their counterclockwise-ordered invariant rays. Each toric model of $U$ determines such a factorization, but in general, different factorizations may correspond to toric models of non-deformation-equivalent log Calabi--Yau surfaces. Theorem \ref{sing} shows that this does not happen when $\mu^{-1}$ is one of Kodaira's monodromies.
\end{rmk}

\begin{eg}In Example \ref{cubicdevelop}, we have $\delta^1(v_1) = (-1,0)$ and $\delta^1(v_2) = (0,-1)$, so we thus see that the monodromy for the cubic surface is $-\id$.
\end{eg}

\begin{eg} \label{M05 mono}Similarly, for Example \ref{M05} we have $\delta^1(v_1) = (1,-1)$ and $\delta^1(v_2) = (1,0)$, so the monodromy is
\begin{gather*}
	\mu= \left(\begin{matrix}
		1 &1 \\
		-1 &0
		\end{matrix}\right)^{-1} =
		\left(\begin{matrix}
		0 &-1 \\
		1 &1
		\end{matrix}\right)
\end{gather*}
with respect to the basis $\big\{ \delta^0(v_1) = (1,0) , \delta^0(v_2) = (0,1) \big\}$.
\end{eg}

We have that $U^{\trop}$ is uniquely determined (as an integral linear manifold, up to isomorphism)
 by its monodromy, and that a factorization of the monodromy into unipotent elements with cyclically ordered eigenrays as above corresponds to a toric model for a Looijenga pair (up to deformation), and hence to a seed as in Section~\ref{Looijenga seeds}. By ``eigenray,'' we mean an eigenline with a chosen direction.

\subsubsection{Mutations and monodromy}\label{mutmono}

We now describe the monodromy of $U^{\trop}$ directly in terms of seed data. Use $\mu_{i,S}$ to indicate that we are mutating a seed $S$ with respect to a vector $e_i$. We consider the induced map on $\?{N_2}$, identified with $N_{\?{Y}}$ as in Section~\ref{Looijenga seeds}, which we denote by $\?{\mu}_{i,S}$. This is not hard to describe~-- it is given by equation~\eqref{seedmut}, with each $e_i$ replaced by $v_i:=p_2^*(e_i)$, and $(\cdot,\cdot)$ replaced by the induced non-degenerate bilinear form $(\cdot\wedge\cdot)$ on $N_{\?{Y}}$.
 Assume that the $v_i$'s are positively ordered with respect to the orientation induced by this form.

Now we observe that, in the notation of equation~\eqref{genmut}, $\?{\mu}_{i,S}^2=\mu_{v_i}^{-d_i'}$. Thus, the inverse monodromy $\mu^{-1}$ of $U^{\trop}$ is $\mu^{-1}= \prod \?{\mu}_{i,S}^2$, where the product is taken over all $i$, with the $v_i$'s being ordered counterclockwise as we move from right to left in the product. Note that the $v_i$'s in this formula are not affected by the previous mutations!

Alternatively, by equation~\eqref{hardmu}, we have $\mu = \?{\mu}_{n,S^{n}}^{-2} \circ \?{\mu}_{n-1,S^{n-1}}^{-2} \circ \cdots \circ \?{\mu}_{1,S^{1}}^{-2}$, where $S^1:=S$, and $S^{k}:= \mu_{k-1,S^{n-1}}^{-2} \big(S^{k-1}\big)$. That is, we apply the inverse mutation twice with respect to one vector, then twice with respect to the next vector in the new seed, and so on.

This straightforward way to compute the monodromy is potentially useful because in Section~\ref{classification} we classify cluster varieties in terms of their monodromies (among other things).

\subsection[Lines in $U^{\trop}$]{Lines in $\boldsymbol{U^{\trop}}$} \label{lines}

For us, a {\it line} $L$ in $U^{\trop}$ will simply mean the image of a linear map $L\colon \bb{R}\rar U^{\trop}_0$ (we abuse notation by letting $L$ denote the map and its image). A line together with such a choice of linear map will be called a {\it parametrized line}.

The {\it signed lattice distance} of a parametrized line $L$ from the origin is given by the skew-form $L(t)\wedge L'(t)$, where we use the canonical identification of the vector from $0$ to $L(t)$ with a vector in $T_{L(t)}$. Note that the lattice distance does not depend on $t$. We will write $L^{>0}$ to denote that a~line $L$ has positive lattice distance from the origin (i.e., goes counterclockwise about the origin), or~$L^{<0}$ to denote that it has negative lattice distance from the origin.

We will say that a parametrized line $L$ {\it goes to infinity parallel to} $q$ if, for any open cone $\sigma \ni q$, there is some $t_\sigma\in \bb{R}$ such that $t>t_{\sigma}$ implies $L(t) \in \sigma$, $L'(t)=q$ under parallel transport in~$\sigma$. Similarly for {\it coming from infinity parallel to} $q$, with $t>t_{\sigma}$ replaced by $t<t_{\sigma}$ and $L'(t)=q$ replaced with $-L'(t)=q$.

We let $L(\infty)$ and $L(-\infty)$ denote the directions in which $L$ goes to and comes from infinity. We use the subscript $q$ to indicate that a line $L$ goes to infinity parallel to $q$. For example, $L_q^{>0}$~denotes a line which goes to infinity parallel to~$q$ with the origin on its left.

We say that an unparametrized line goes to infinity parallel to $q$ if it admits a parametrization which does. In general, a line need not go to infinity at all. In fact, one characterization of $U$ being positive is that every line both goes to and comes from infinity, cf.\ Section~\ref{PositiveCase}.

We note that the monodromy about the origin in $U^{\trop}$ allows lines to wrap around the origin and self-intersect. We say that a line $L$ {\it wraps} if it intersects every ray, except possibly one, at least once. It wraps $k$ times if it hits each ray at least $k$ times, except possibly for one ray which it may hit only $(k-1)$ times.

\begin{eg} If $(Y,D)$ is the cubic surface introduced in Example \ref{cubicdevelop}, then for any ray $\rho \subset {U^{\trop}}$, $U^{\trop}\setminus \rho$ is isomorphic as an integral linear manifold to an open half-plane. Both ends of any line will go to infinity in the same direction. If we now make a non-toric blowup on some $D_{\rho_q}$, then in the new integral linear manifold, any line will self-intersect unless both ends will go to infinity parallel to $q$.
\end{eg}

\subsection[Some integral linear automorphisms of $U^{\trop}$]{Some integral linear automorphisms of $\boldsymbol{U^{\trop}}$}\label{LinearAuts}

Assume that $U$ is positive, so lines go to infinity on both ends. Given a point $q$ in $U^{\trop}$, define
\begin{gather*}
	\nu_+(q):=L_q^{>0}(-\infty), \hspace{.25 in} \nu_-(q):=L_q^{<0}(-\infty).
\end{gather*}
Intuitively, both operations correspond to ``negating'' a vector in the integral linear manifold, but using different choices of charts. These clearly lift to maps $\wt{\nu}^+$ and $\wt{\nu}^-\colon \wt{U}_0^{\trop}\rar \wt{U}_0^{\trop}$, which may be viewed as rotation $180^{\circ}$ clockwise or counterclockwise, respectively.

\begin{lem}\label{nu}$\nu_+$ and $\nu_-$ are integral linear and inverse to each other.
\end{lem}
\begin{proof}
This follows from $\wt{\nu}^{\pm}$ being integral linear and inverse to each other, which is clear since clockwise/counterclockwise $180^{\circ}$-rotations of $\bb{R}^2$ are integral linear and inverse to each other.
\end{proof}

We will see in Proposition \ref{nuAut} that $\nu_{\pm}$ are induced by $\Gamma$.

\subsection{Useful facts from \cite{Man1}}\label{Man}
The following is a restatement of a Lemmas 3.7 and Corollary 3.8 from \cite{Man1}:

\begin{lem}\label{model}
Let $L\subset U^{\trop}$ be a line which does not wrap. Let $u$ and $v$ be the directions in which $L$ goes to infinity. Let $\sigma_L\subset U^{\trop}$ be the closed cone which is bounded by $u$ and $v$ and which does not contain any points of $L$. Then some compactification of $U$ admits a toric model whose non-toric blowups are all along divisors corresponding to rays in $\sigma_L$. Furthermore, after choosing a suitable compactification of $U$, there is only one such toric model with blowups centered on divisors corresponding to rays in $\sigma_L \setminus \rho_u$ $($alternatively, $\sigma_L \setminus \rho_v)$.
\end{lem}

Suppose $U$ is positive, and let $(Y,D)$ be a compactification of $U$ with $D$ supporting an effective $D$-ample divisor (cf.\ Section~\ref{Looijenga seeds}). Let $\NE(Y)$ denote the cone of effective curve classes of $Y$. \cite{GHK1} constructs a flat family $\s{V}\rar \Spec \kk[\NE(Y)]$ mirror to $U$ which admits a canonical $\kk[\NE(Y)]$-module basis of theta functions $\{\vartheta_q\}_{q\in U^{\trop}(\bb{Z})}$. \cite{GHK2} then shows that $U$ can be realized as a fiber of $\s{V}$, thus giving theta functions on $U$. Recall from Section~\ref{ClusterDefs} that a global monomial is regular function on $\s{X}$ whose restriction to some seed $\s{X}$-torus is a monomial. We also call the restriction to a fiber $U\subset \s{X}$ of such a function a global monomial. Section~3.6 of \cite{Man1} observes the following (phrased differently):

\begin{lem}\label{GlobalMonomial}Take $\sigma_L$ as in Lemma~{\rm \ref{model}}. For any $q\in \sigma_L$, $\vartheta_q$ is a global monomial.
\end{lem}

Assume $U$ is positive, and let $V$ denote a generic fiber of the mirror $\s{V}$. Given a~line \mbox{$L\subset V^{\trop}$}, let $Z(L)$ denote the connected component of $V^{\trop}\setminus L$ which contains the origin. For $q\in U^{\trop}(\bb{Z})$, $v\in V^{\trop}(\bb{Z})$, we can define $\vartheta_q^{\trop}(v):=\val_{D_v}(\vartheta_q)$, where $D_v$ is the boundary divisor corresponding to $v$ in some compactification of~$V$. \cite{Man1} extends $\vartheta_q^{\trop}$ to all of $V^{\trop}$ and describes its fibers explicitly. In particular Corollary~4.11 of~\cite{Man1} implies:
\begin{lem}\label{NegativeFibers}
For each $d<0$ and $q\in U^{\trop}(\bb{Z})$, the set $\big\{d<\vartheta_q^{\trop}<0\big\}\subset V^{\trop}$ is equal to~$Z(L)$ for some line $L$. Thus, if every line wraps, then every $\vartheta_q^{\trop}$ is non-positive everywhere, and in fact, $f^{\trop}$ is non-positive everywhere for every regular function on~$V$.
\end{lem}
\begin{proof}
The last statement uses that every regular function is a linear combination of theta functions, and valuations of linear combinations of theta functions are given by taking the minima of the valuations of each term (Remark~4.4 and the preceding paragraph of~\cite{Man1} explain why no cancellations occur).
\end{proof}

\subsection{The tropicalization determines the charge}\label{trop-charge}

One natural question to ask is to what extent $U^{\trop}$ determines $U$. We will see in the next section that in many cases, $U$ is uniquely determined up to deformation by $U^{\trop}$. This is not always the case though: for example, there are two degree $8$ Del Pezzo's with an irreducible choice of anti-canonical divisor which have the same $U^{\trop}$ but are not deformation equivalent. This subsection shows that $U^{\trop}$ does at least determine the number of non-toric blowups.

\begin{dfn}The {\it charge}\footnote{More generally, the charge of a log Calabi--Yau variety $(Y,D=D_1+\dots+D_n)$ is given by $c(Y,D):=\dim(Y)+\rank(\Pic(Y))-n$.} of a Looijenga pair $(Y,D)$ is the number of non-toric blowups in a toric model for some toric blowup of $(Y,D)$.
\end{dfn}

\begin{lem} \label{ch}A Looijenga pair $(Y,D=D_1+\dots+D_n)$ with $n>1$ and intersection matrix $H:= (D_i\cdot D_j )$ has charge
\begin{gather}\label{charge}
	c(Y,D)=12-3n-\Tr(H).
\end{gather}
\end{lem}

\begin{proof}First note that, for $n>1$, toric blowups increase $n$ by $1$, decrease $\Tr(H)$ by $3$, and keep the charge constant, so equation~\eqref{charge} is unaffected by toric blowups and blowdowns. Similarly, non-toric blowups decrease $\Tr(H)$ by $1$ and increase the charge by $1$, so the validity of the equation is also unaffected by non-toric blowups. Since every Looijenga pair is related to a copy of the toric pair $\big(\bb{P}^2,D\big)$ by some sequence of toric blowups, toric blowdowns, and non-toric blowups, it now suffices to just check this case. We have $c\big(\bb{P}^2,D\big)=0$, $n=3$ and $\Tr(H)=3$, so the equation holds.
\end{proof}

A similar formula appears in \cite{GHK2}: $c(Y,D)=12-\big(n+K_Y^2\big)$.

\begin{prop}\label{TropCharge}Suppose that $(Y,D)$ and $(Y',D')$ are two Looijenga pairs with the same tropicalization $U^{\trop}$. Then $c(Y,D)=c(Y',D')$.
\end{prop}

\begin{proof}Let $\Sigma_Y$ and $\Sigma_{Y'}$ be the corresponding fans in $U^{\trop}$. There exists some nonsingular common refinement $\Sigma$ which is the fan for a toric blowup of both $(Y,D)$ and $(Y',D')$. The intersection matrices for these two toric blowups are the same, since each can be determined from $\Sigma$, so the claim follows from Lemma~\ref{ch}.
\end{proof}

\section{Classification}\label{classification}

Here we give several equivalent classifications for the possible deformation classes of Looijenga pairs. These classifications are based on the intersection matrix $H$ of $D$, the intersection form~$Q$ on $D^{\perp}\cong K_2$ (or the restriction of~$Q$ to $D^{\perp}_{\Eff} \subset D^{\perp}$, cf.\ Section~\ref{canonical intersection}), the monodromy~$\mu$ of~$U^{\trop}$, the properties of lines in $U^{\trop}$, the global functions on $U$, the properties of the quiver for a corresponding cluster structure, and various other properties. This may be viewed as a classification of rank-$2$ cluster varieties up to the notion of fiberwise-equivalence given in Definition~\ref{EquivalentClusters}. The classification is not totally new~-- for example, the cases that we refer to as ``no lines wrap'' or ``some lines wrap'' are simply the finite-type or acyclic cases, respectively, in the cluster language. However, we do offer new characterizations of these cases.

Throughout this section, $D$ will be called {\it minimal} if it has no $(-1)$-components. We begin with classifying the negative (semi-)definite cases. Most of the statements for these cases appear in \cite{GHK1} (some only in arXiv v1) or in~\cite{GHK2}, or else follow easily, as we will show.

\subsection{The negative semi-definite cases}

\begin{thm}[the negative definite cases]\label{NegD}
The following are equivalent:
\begin{enumerate}\itemsep=0pt
\item[$1.$] The {intersection matrix} $H=(D_i \cdot D_j)$ is negative definite.
\item[$2.$] Any {developing map} $\delta$ as in Section~{\rm \ref{dev}} embeds the universal cover $\wt{U}_0^{\trop}$ of $U^{\trop}_0$ into a~strictly convex cone in~$\bb{R}^2$.
\item[$3.$] The {monodromy} satisfies $\Tr(\mu) > 2$.
\item[$4.$] All {lines} in $U^{\trop}$ wrap infinitely many times around the origin, meaning that they hit each ray infinitely many times. However, none of the lines are circles.
\item[$5.$] The {quadratic form} $Q$ on $D^{\perp}$ is not negative semi-definite.
\item[$6.$] $U$ and its deformations admit no non-constant global regular {functions}.
\item[$7.$] $D$ can be blown down to get a surface $\?{Y}$ with a cusp singularity. If $D$ is minimal, $D_i^2 \leq -2$ for all $i$, and $D_i^2 \leq -3$ for some~$i$.
\end{enumerate}
\end{thm}

\begin{proof}The equivalence of (1) and the cusp singularity statement from (7) is taken as the definition of a cusp singularity in \cite{GHK1}, while the statement for $D$ minimal appears in \cite[Example~1.10]{GHK1} and is easily checked. The equivalence of (1) and (2) is \cite[Lemma~1.5, arXiv v1]{GHK1}.

(7)$\Rightarrow$(6) is in \cite{GHK2} and can be seen from Hartog's Lemma since $\?{Y}$ is compact and $\?{Y}\setminus U$ is an isolated singularity. For (6)$\Rightarrow$(1), suppose $D$ is not negative definite. If $D$ is negative semi-definite, then by Theorem~\ref{NegSD}(7) below, some deformation of $(Y,D)$ admits an elliptic fibration over $\bb{P}^1$ with $D$ as a fiber, the restriction of this fibration to the deformed $U$ gives a~non-constant global regular function. On the other hand, if $D$ is not negative semi-definite, then $\Spec \Gamma(U,\s{O}_U)$ is two-dimensional by Theorem \ref{Pos}(6) below (i.e., by \cite[Lemma~6.9]{GHK1}), and so $U$ has many non-constant global sections.

To see that (2)$\Rightarrow$(3), recall that for $\mu \in \SL_2(\bb{R})$, $|\Tr(\mu)|\leq 2$ if and only if $\mu$ is a rotation or shear map, and either of these would contradict the claim that the image of~$\delta$ lies in a~strictly convex cone (recall that $\mu^{-1}$ acts by deck transformations, cf.\ Example~\ref{focusfocus}). Similarly, if \mbox{$\Tr(\mu)<-2$}, then the eigenvalues are negative, and this also contradicts the strict convexity. So $\Tr(\mu)>2$ is the only remaining possibility. Conversely, $\Tr(\mu)>2$ implies that $\mu$ acts by hyperbolic rotation, and the geometric interpretation of this makes it clear that the image of $\delta$ must be contained in a~strictly convex cone, giving (3)$\Rightarrow$(2).

Note that (4) is equivalent to the statement that every line in $\bb{R}^2$ which does not pass through the origin and which intersects interior of the image of $\delta$ must hit the boundary of the image of $\delta$ (circles would be parallel to both boundary rays, impossible unless the image of $\delta$ is a~half-space). The equivalence of (2) and (4) is clear from this.

The equivalence of (1) and (5) is in~\cite{GHK2}. For the direction (1)$\Rightarrow$(5), note that if both~$H$ and~$Q$ were negative semi-definite, then all of $\Pic(Y)$ would be negative semi-definite. But we can assume by possibly making some additional toric blowups that $Y$ is a blowup of a projective toric variety, and then the pullback of an ample class will have positive self-intersection, a~contradiction. The converse (5)$\Rightarrow$(1) is an easy consequence of the Hodge index theorem.
\end{proof}

We now consider the cases where $H$ is negative semi-definite, but not negative definite. Once again, the following statements mostly appear in \cite{GHK1} and \cite{GHK2}, or else follow easily.
\begin{thm}[the strictly negative semi-definite cases]\label{The strictly negative semi-definite case}\label{NegSD}
The following are equivalent:
\begin{enumerate}\itemsep=0pt
\item[$1.$] The {intersection matrix} $H$ is negative semi-definite but not negative definite.
\item[$2.$] Any {developing map} $\delta$ for $U^{\trop}_0$ identifies the universal cover of $U^{\trop}_0$ with a half-plane in~$\bb{R}^2$.
\item[$3.$] The monodromy $\mu$ is $SL_2(\bb{Z})$-conjugate to a matrix of the form
	$\left(\begin{smallmatrix}
		1 &a \\
		0 &1
		\end{smallmatrix}\right)$,
		 with $a > 0$.
\item[$4.$] Some {lines} in $U^{\trop}$ are circles $($all others wrap infinitely many times around the origin$)$.
\item[$5.$] If $D$ is minimal, then $D \in D^{\perp}$, or equivalently, either $D_i^2 = -2$ for all $i$, or~$D$ is irreducible with $D^2=0$.
\item[$6.$] The {quadratic form} $Q$ on $D^{\perp}$ is negative semi-definite but not negative definite.
\item[$7.$] For $D$ minimal, $(Y,D)$ is deformation equivalent to a Looijenga pair $(Y',D')$ which admits an elliptic fibration having~$D'$ as a fiber.
\end{enumerate}
\end{thm}
\begin{proof}
The equivalence (1)$\Leftrightarrow$(2) is part of \cite[Lemma 1.5, arXiv v1]{GHK1}. The equivalence (2)$\Leftrightarrow$(3) is straightforward using the characterization of matrices with trace $\pm 2$ as shear matrices.

For (2)$\Leftrightarrow$(4), note that circles correspond to lines in the half-space $\delta\big(U_0^{\trop}\big)$ which are parallel to the boundary of this half-space. Any other line in $\bb{R}^2$ which intersects the interior of $\delta\big(U_0^{\trop}\big)$ must also hit the boundary of $\delta\big(U_0^{\trop}\big)$, and as in the proof of Theorem~\ref{NegD}, these lines wrap infinitely many times around the origin.

The equivalence (1)$\Leftrightarrow$(5) is easily checked and was stated in \cite[Section~0.5.2, arXiv v1]{GHK1}. Note that if $D$ is minimal, reducible, and has intersecting components $D_{i-1}, D_i$ with $D_i$ of self-inter\-section greater than $(-2)$ (hence $\geq 0$ by minimality of $D$) and $D_{i-1}^2\leq 0$, then $\big[\big({-}D_{i-1}^2+1\big)D_i+D_{i-1}\big]^2\geq 2-D_{i-1}^2 > 0$.

The equivalence (1)$\Leftrightarrow$(6) is also due to~\cite{GHK2}. Let us first check (1)$\Rightarrow$(6). By the statement (5)$\Rightarrow$(1) in Theorem~\ref{NegD}, we know that (1) here implies $Q$ is negative semi-definite. Then by~(5), if~$D$ is minimal, we have $D\cdot D_i=0$ for each $i$, hence $D\in D^{\perp}$ and $Q(D)=0$, thus showing that~$Q$ is not definite. If~$D$ is not minimal, let $p\colon (Y,D)\rar (Y',D')$ be a sequence of toric blowdowns to a case with minimal boundary. Then $p^*(D')$ is in $D^{\perp}$ and satisfies $Q(D')=0$.

For (6)$\Rightarrow$(1), suppose $H$ were not negative semi-definite. Then some divisor $C$ supported on~$D$ has positive self-intersection, and by the Hodge index theorem, $C^{\perp}$ (which contains $D^{\perp}$) is negative definite. So if (6) holds, then $H$ must be negative semi-definite. But by the statement (1)$\Rightarrow$(5) of Theorem \ref{NegD}, (6) implies that $H$ is not negative definite, giving (6)$\Rightarrow$(1)

Finally, (7) implies that $D\cdot D_i=0$ for each $i$, hence $D$ must have the form of (5). The converse was stated in \cite[Section~0.5.2, arXiv v1]{GHK1}, and it can be seen as follows: Let $E$ be any $(-1)$-curve hitting a component $D_i\subset D$ (e.g., an exceptional divisor from a toric model). Let~$\?{D}$ be the push-forward of $D$ after blowing down $E$, and similarly for components of $D$. Then $\?{D}^2=\?{D}\cdot \?{D}_i=1$, while $\?{D}\cdot \?{D}_j=0$ for $j\neq i$. By Riemann--Roch, the linear system~$|\?{D}|$ has dimension at least $1$. Let $C$ be any curve in $|\?{D}|$ other than $\?{D}$ itself. From the intersection numbers, we see that $C$ must be disjoint from $\?{D}$ except for a single point of intersection $p$ in the interior of $\?{D}_i$. Blowing up $p$ results in a new Looijenga pair $(Y',D')$, with $D'$ the proper transform of $D$, such that the proper transform of $C$ is linearly equivalent to and disjoint from~$D'$. Thus, $|D'|$ must contain a pencil giving an elliptic fibration of $Y'$.
\end{proof}

As stated above, if $D$ is minimal then it is either irreducible or consists of $n>1$ $(-2)$-curves. The largest possible $n$ here is $9$. This follows from Lemma~\ref{ch}, which says that the charge is $c(Y,D)=12-3n-\Tr(H)=12-n$. The charge is by definition non-negative, giving us $n\leq 12$. Furthermore, the classifications below then imply that some lines do not wrap if $c(Y,D)\leq 2$, so then $n\leq 9$. A case with $n=9$ can be explicitly constructed.

\subsection{The positive cases}\label{PositiveCase}

Several characterizations of the positive cases appear in \cite[Lemma~6.9]{GHK1}. We state some of these and others now.
\begin{thm}[the positive cases]\label{Pos} The following are equivalent:
\begin{enumerate}\itemsep=0pt
\item[$1.$] The {intersection matrix} $H$ is not negative semi-definite.
\item[$2.$] The {developing map} for $U^{\trop}_0$ is not injective.
\item[$3.$] {Lines} in $U^{\trop}$ wrap at most finitely many times, so both ends of each line go to infinity.
\item[$4.$] The {quadratic form} $Q$ on $D^{\perp}$ is negative definite.
\item[$5.$] $U$ is deformation equivalent to an affine surface.
\item[$6.$] $U$ is a minimal resolution of $\Spec (\Gamma(U,\s{O}_U))$, which is an affine surface with at worst Du Val singularities.
\item[$7.$] $D$ supports a $D$-ample divisor.
\end{enumerate}
\end{thm}
\begin{proof}The equivalences involving (1), (2), (3), and (4) follow from negating the corresponding statements in Theorems~\ref{NegD} and~\ref{NegSD}. The equivalence of (1), (6), and (7) is \cite[Lemma~6.9(1.1--1.3)]{GHK1}. For (6)$\Rightarrow$(5), we just need to know that we can deform $U$ to have no $(-2)$-curves, as these are the exceptional divisors of the minimal resolution in (6). This follows from \cite[Proposition~4.1]{GHK_MLP}. Conversely, if a deformation of $U$ is affine, then $\Spec (\Gamma(U,\s{O}_U))$ must be two-dimensional, and so we cannot be in a negative semi-definite case by Theorems~\ref{NegD}(6) and~\ref{NegSD}(7).
\end{proof}

If any of these conditions hold, we say that $U$ is {\it positive}. We have several sub-cases:

\begin{thm}[all lines wrap/positive non-acyclic cases]\label{all-wrap}
The following are equivalent:
\begin{enumerate}\itemsep=0pt	
 \item[$1.$] {Lines} in $U^{\trop}$ all wrap, but only finitely many times. 
 \item[$2.$] Every sheet of the {developing map} is convex, but the developing map is not injective. 
 \item[$3.$] Non-zero global regular {functions} on $U$ are not generically $0$ along any boundary divisor of any compactification $(Y,D)$ of $U$ $($i.e., the corresponding valuations are non-positive$)$. On the other hand, there are enough global regular functions that $\dim \Spec \Gamma(U,\s{O}_U)=2$.
 \item[$4.$] The inverse {monodromy matrix} $\mu^{-1}$ is conjugate to a Kodaira matrix\footnote{In~\cite{Kod}, Kodaira listed the matrices which can appear as monodromies about singular fibers of elliptic fibrations of surfaces. See Tables~\ref{wraptable} and~\ref{tab:no wrap} for a list of these matrices.} of type $I_k^*$, $II^*$, $III^*$, or $IV^*$. 
 \item[$5.$] If $D$ is minimal, then either $D=D_1+D_2$ with $D_1^2 = 0$ and $-1\neq D_2^2 \leq 0$ $($up to re-labelling$)$, or $D$ is irreducible with $1\leq D^2 \leq 4$. 	
 \item[$6.$] $U$ can be constructed from $\big(\bb{P}^2,D\big)$, with $D=D_1+D_2+D_3$ a triangle of lines, by blowing up~$d_i$ times on~$D_i$ for each~$i$, with $(d_1,d_2,d_3)$ as in the final column of Table~{\rm \ref{wraptable}}. Equivalently, $U$ comes from a seed with $E=(e_1,e_2,e_3)$, $F=\varnothing$, $\langle \cdot , \cdot \rangle = 	\left(\begin{smallmatrix}
		0 &1 &-1 \\
	 -1 &0 &1 \\
	 1 &-1 &0
		\end{smallmatrix}\right)$, and multipliers $(d_1,d_2,d_3)$ as in the final column of Table~{\rm \ref{wraptable}}.
\item[$7.$] $D^{\perp}_{\Eff}=D^{\perp}$, and the {quadratic form} $Q$ is of type $D_n$ $(n\geq 4)$ or $E_n$ $(n=6$, $7$, or~$8)$. 
\end{enumerate}
\end{thm}

\begin{proof}(1)$\Leftrightarrow$(2) is clear from the definitions. (1)$\Leftrightarrow$(3) follows immediately from Lemma \ref{NegativeFibers} (the ring of global regular functions being two-dimensional is equivalent to positivity).

For (1)$\Rightarrow$(5), using the construction of $U^{\trop}$ from charts in \eqref{linear}, we can easily see that having any $D_i^2>0$ with $D$ not irreducible would allow a line to not wrap. On the other hand, having every $D_i^2 \leq -2$ would mean we are in a negative semi-definite case. So if $D$ is minimal and not irreducible, then $D_i^2$ must be $0$ for some $i$. $D$ having more than one additional component would allow a non-convex sheet of the developing map, so the claim follows, except for when $D$ is irreducible. If $D$ is irreducible and $D^2 > 4$, then the proper transform of $D$ after taking a~toric blowup would have positive self-intersection, which we have already ruled out, and $D^2<1$ would mean we are in a negative semi-definite case.

For (5)$\Rightarrow$(2), observe that in the $D_1^2=D_2^2=0$ case, every sheet of any developing map is convex (but not strictly convex). The other cases come from non-toric blowups and toric blow-downs of this, so the sheets of their developing maps will of course still be convex (non-toric blowups make these sheets ``more convex'').

(5)$\Rightarrow$(4) is a straightforward check, as is (4)$\Rightarrow$(2). We now have the equivalence of~(1) through~(5).

(6)$\Rightarrow$(7) is also straightforward. For $U$ generic, $D^{\perp}$ is generated by classes of the form $E_{i,j_1}-E_{i,j_2}$ (where $E_{i,j}$ denotes the exceptional divisor from a non-toric blowup on $D_i$), together with a class of the form $L-E_{1,j_1} - E_{2,j_2} - E_{3,j_3}$, where $L$ is the class of a generic line in $\bb{P}^2$. If we choose all the blowup points on each $D_i$ to be infinitely near, and choose the blowup points on different $D_i$'s to be colinear, then $D^{\perp}$ is generated by effective divisors with the correct intersections.

For (7)$\Rightarrow$(1), $Q$ of type $D_n$ or $E_n$ implies that $Q$ is negative definite, and so Theorem \ref{Pos}(4) tells us that we are not in an $H$ negative semi-definite case. We also cannot be in a some-lines-wrap or no-lines-wrap case because, as we see below, $Q|_{D^{\perp}_{\Eff}}$ in these cases is a direct sum of~$A_{n_i}$'s.

It now suffices to show that (5)$\Rightarrow$(6) (since (4)$\Leftrightarrow$(5), this means we are showing that $U^{\trop}$ really does determine the deformation type of $U$ in these cases). For the $I_0^*$ case, we have $\mu^{-1}=-\Id$. Such a $U^{\trop}$ contains a reflexive polytope\footnote{By a reflexive polytope in $U^{\trop}$, we mean an integral polytope in the sense of \cite[Definition~2.10]{Man1} which is strongly convex in the sense of \cite[Definition~5.4]{Man1} whose only interior integral point is $0$.} with $3$ integral points on the boundary, and this implies that $U$ must be an affine cubic surface (cf.\ Example~5.21 in~\cite{Man1}), which we know can be obtained as in Example~\ref{cubicdevelop}.

Now for the $I_k^*$ cases, we can choose a compactification $(Y,D)$ of $U$ with $D_1^2=D_2^2=-1$ and $D_3^2=-1-k$. The divisor $C:=D_1+D_2$ has $C\cdot D_1=C\cdot D_2 = C^2 = 0$, and $C\cdot D_3 = 2$. By Riemann--Roch, $\dim |C|\geq 1$. If $C$ is the only singular element of some pencil $\bb{P}^1\subset |C|$, then (for~$U$ generic in its deformation class) $Y\setminus C$ is a $\bb{P}^1$-bundle over $\bb{A}^1$, hence has Euler characteristic $2$. So then $Y$ has Euler characteristic $5$. However, we know from Section~\ref{trop-charge} that~$U^{\trop}$ determines the charge $c$ of $(Y,D)$, which in this situation is $6+k$. One checks that the Euler characteristic of a Looijenga pair with $n$ boundary components and charge $c$ is $n+c$, which in this case is $9+k > 5$. So $|C|$ must contain other singular curves. These must contain irreducible rational components $E_1$, $E_2$ with $E_i\cdot D_3 = 1$ and $E_i^2=-1$. Blowing down either of these is a non-toric blowdown and reduces us to the $I_{k-1}^*$ case, so the claim follows by induction.

For the $IV^*$ case, we have a compactification of $U$ with $D=D_1+D_2+D_3$, $D_1^2=-1$, $D_2^2=D_3^2 = -2$. Note that $D\cdot D_1=1$, while $D\cdot D_2=D\cdot D_3 = 0$, so $\dim |D|\geq 1$. Thus, there is some point on $D_1$ which we can blow up to get a new pair $\big(\wt{Y},\wt{D}\big)$, with exceptional divisor~$E$, such that $\wt{Y}$ admits an elliptic fibration with $\wt{D}$ being a fiber and~$E$ being a~section. Such a surface can be obtained by blowing up $9$ base-points for a pencil of cubics in~$\bb{P}^2$, with~$E$ being the exceptional divisor of the final blowup (cf.~\cite{HL}). $\wt{D}$ then is the proper transform of one of the cubics $\?{D}$ in the pencil, so there must have been $3$ base-points on each component~$\?{D}_i$ of~$\?{D}$. Thus, after blowing~$E$ down, we see that~$Y$ must contain disjoint $(-1)$-curves hitting each component of~$D$. Blowing down a $(-1)$-curve hitting, say, $D_2$, reduces to the $I_1^*$ case we have already dealt with.

A similar argument works for the $III^*$ case using a compactification of $U$ with $D=D_1+D_2$, $D_1^2=-1$, $D_2^2 = -2$, and blowing up a point in~$D_1$ to get a surface with an elliptic fibration. The $II^*$ case is also similar, using~$D$ irreducible with self-intersection $1$ and blowing up some point in~$D$ to get a surface with an elliptic fibration.
\end{proof}

Table \ref{wraptable} summarizes the different cases from the Theorem~\ref{all-wrap} above.

\begin{table}[h] \centering
 \begin{tabular}{ | l | l | l | l | l |}
 \hline
 Kodaira matrix & Cartan form $Q$ & Monodromy $\mu$ & $(d_1,d_2,d_3)$ \\ \hline
 $I^*_{k}$ ($k\geq 0$) & $D_{k+4}$ & 	 $\left(\begin{matrix}
		-1 &k \\
	 0 &-1
		\end{matrix}\right)$				& $(2,2,2+k)$ 	 \\ \hline
 $IV^*$ & $E_6$ & 		$\left(\begin{matrix}
		0 &1 \\
	 -1 &-1
		\end{matrix}\right)$		 & $(2,3,3)$	 \\ \hline
 $III^*$ &$E_7$ &	$\left(\begin{matrix}
		0 &1 \\
	 -1 &0
		\end{matrix}\right)$	 	 & $(2,3,4)$	 \\ \hline
 $II^*$ &$E_8$ & 	$\left(\begin{matrix}
		1 &1 \\
	 -1 &0
		\end{matrix}\right)$				& $(2,3,5)$		 \\ \hline
 \end{tabular}\caption{Cases where all lines wrap.}\label{wraptable}
 \end{table}

\begin{thm}[not all lines wrap/acyclic cases]\label{not-all-wrap}
The following are equivalent:
\begin{enumerate}\itemsep=0pt
\item[$1.$] $U^{\trop}$ contains a {line} which does not wrap.
\item[$2.$] Some compactification of $U$ admits a toric model $Y\rar \?{Y}$ for which all the non-toric blowups are on divisors corresponding to rays in one half of $N_{\?{Y}}$. I.e., there is some seed~$S$ for which all of the non-frozen vectors' images in~$p_2^*(N)$ lie in one half of the plane.
\item[$3.$] Cluster varieties corresponding to~$U$ are {acyclic}.
\item[$4.$] The intersection of the Langlands dual {cluster complex} $\s{C}\subset \s{X}^{\trop}$ with $U^{\trop}$ is non\-empty.
\item[$5.$] There exists a {global monomial} on~$U$.
\item[$6.$] The {quadratic form} $Q$ on $D^{\perp}$ is negative definite, and $Q|_{D^{\perp}_{\Eff}}$ is a direct sum of $A_{n_i}$'s. In fact, it is $A_{d'_1-1}\oplus \cdots \oplus A_{d'_m-1}$, where the $(d'_i)$'s are the modified multipliers for a coprime seed corresponding to~$U$ $($equivalently, $d_i'$ is the number of non-toric blowups on~$D_i$ in a~toric model for a~compactification~$U)$.
\end{enumerate}
\end{thm}
\begin{proof}(1)$\Leftrightarrow$(2) is Lemma \ref{model}. (2)$\Leftrightarrow$(3) was observed in Section~\ref{quivers}.

For (2)$\Leftrightarrow$(4), note that for some seed vector $e_i$ for a seed $S$, the set $\{e_i\geq 0\}\cap U^{\trop}$ is the same as the set $(v_i\wedge \cdot) \geq 0$, where $\wedge$ is the symplectic form on $U^{\trop}$ induced by $[ \cdot,\cdot]$. The intersection of these positive half-spaces for all non-frozen $e_i$'s is clearly nonempty if and only if~$S$ is as in~(2).

 (1)$\Rightarrow$(5) follows from Lemma \ref{GlobalMonomial}. For (5)$\Rightarrow$(1), note that for a global monomial $\vartheta_q$, the tropicalization $\vartheta_q^{\trop}$ is positive somewhere, and so Lemma \ref{NegativeFibers} implies that the fibers $\vartheta_q^{\trop}=d<0$ are lines which do not wrap.

 (6)$\Rightarrow$(1) because if every line does wrap (possibly infinitely many times), then we have seen that either $Q$ is not negative-definite or $Q|_{D^{\perp}_{\Eff}}$ is of type $D_n$ or $E_n$.

 For (2)$\Rightarrow$(6), first note that $Q$ is negative definite on $D^{\perp}$ by positivity of $U$. Now, let $(Y,D)\rar \big(\?{Y},\?{D}\big)$ be the toric model corresponding to a seed with all non-toric blowups corresponding to rays in one half of the plane $N_{\?{Y}}$. For any curve $\?{C}$ in $\?{Y}$, $\sum \big(\?{C}\cdot \?{D}_i\big) v_i=0$ where $v_i$ is the primitive vector in $N_{\?{Y}}$ corresponding to $\?{D}_i$. If $\?{C}$ is the image of an irreducible effective curve $C\in D^{\perp}$, then $\?{C}\cdot \?{D}_i \geq 0$ for all $i$, and $\?{C}\cdot \?{D}_i$ can only be positive if there is a non-toric blowup point somewhere in $\?{C} \cap \?{D}_i$. Thus, each $\?{C}\cdot \?{D}_i$ must actually be $0$, so $C$ must have been supported on an exceptional divisor. Thus, $D^{\perp}_{\Eff}$ is generated by classes obtained by taking the $d_i'$ blowups to be infinitely near, and then taking the $d_i'-1$ exceptional divisors which do not intersect $D$.
\end{proof}

Let $\s{C}_U$ denote the union of all cones $\sigma_L$ for lines $L$ which do not wrap, where $\sigma_L$ is defined as in Lemma~\ref{model}. We note that the argument for (2)$\Leftrightarrow$(4) above can be modified to prove the following:
\begin{prop} $\s{C}_U$ is the intersection of the Langlands dual cluster complex $\s{C}$ with $U^{\trop}\subset \s{X}^{\trop}$.
\end{prop}
This justifies \cite{Man1} calling $\s{C}_U$ the cluster complex.

\begin{thm}[no lines wrap/finite-type cases]\label{no-wrap}
The following are equivalent:
\begin{enumerate}\itemsep=0pt
\item[$1.$] No {lines} in $U^{\trop}$ wrap.
\item[$2.$] No sheet of the {developing map} is convex.
\item[$3.$] The Laurent phenomenon holds for the $\s{X}$-space, meaning that each $X_i$ is a global monomial. Furthermore, the {global monomials} form an additive basis for the global function on~$U$.
\item[$4.$] The inverse {monodromy} matrix $\mu^{-1}$ is a Kodaira matrix of type $I_k$, $II$, $III$, or~$IV$.
\item[$5.$] Cluster structures for $U$ are of {finite type}, meaning that they have only a finite number of distinct seeds.
\item[$6.$] For some equivalent maximally factored seed, the corresponding quiver $($after removing frozen vectors$)$ is of type $A_1^k$ $(k\in \bb{Z}_{\geq 0})$, $A_2$, $A_3$, or $D_4$.
\item[$7.$] The Langlands dual {cluster complex} $\s{C}\subseteq \s{X}^{\trop}$ contains all of~$U^{\trop}$, and in fact is all of~$\s{X}^{\trop}$.
\end{enumerate}
\end{thm}
\begin{proof}(1)$\Leftrightarrow$(2) is obvious. (1)$\Leftrightarrow$(3) follows from Lemma~\ref{GlobalMonomial}.

To see that (1) implies (5), we need Lemma~\ref{model}, which says that for any line~$L_q^{d<0}$ which does not wrap, there are only finitely many $(-1)$-curves hitting boundary divisors corresponding to rays in the cone $\sigma_{L}$ bounded by $L_q^{d<0}(\pm \infty)$. Since no lines wrap, we can cover $U^{\trop}$ by finitely many cones of the form $\sigma_{L}$, and so there are only finitely many $(-1)$-curves in $Y$ hitting the boundary. Since seeds correspond to certain finite subsets of this collection of $(-1)$-curves, the claim follows.

(5)$\Leftrightarrow$(6) follows from a well-known result of \cite{FZ2}, which says that a cluster algebra is of finite type if and only if the matrix $(-|\epsilon_{ij}|+2\delta_{ij})_{i,j\in I\setminus F}$ is a finite type Cartan matrix. One easily checks that the only quivers of this type which produce rank $2$ cluster varieties are those listed in the statement of theorem, along with types $B_2$, $B_3$, $C_3$, and $G_2$, which are equivalent to types~$A_3$,~$D_4$,~$D_4$, and $D_4$ again, respectively, in the sense of Definition \ref{EquivalentClusters}.

 One can easily check (6)$\Rightarrow$(4) by explicit computations: the $A_1^k$, $A_2$, $A_3$, and $D_4$ quivers correspond to the $I_k$, $II$, $III$, and $IV$ matrices, respectively. (4)$\Rightarrow$(1) is now automatic.

For (5)$\Leftrightarrow$(7), recall that seeds are in bijection with cones of the cluster complex. For any boundary wall $W$ of any cone in $\s{C}$, both sides of $W$ will always be in $\s{C}$, so if there are only finitely many cones, then~$\s{C}$ must fill up all of~$\s{X}^{\trop}$. Conversely, if there are infinitely many cones, then they must ``bunch up'' near some ray $\rho$ which is not in $\s{C}$.
\end{proof}

Table \ref{tab:no wrap} lists the cases where no lines wrap, along with their basic properties. We once again use the notation $(d_1,d_2,d_3)$ to indicate that such a Looijenga pair can be obtained by starting with the toric variety $\big(\bb{P}^2,D=D_1+D_2+D_3\big)$, and then blowing up $d_1$, $d_2$, and $d_3$ points on~$D_1$,~$D_2$, and~$D_3$, respectively.
\begin{table}[h]\centering
 \begin{tabular}{ | l | l | l | l | l|}
 \hline
 Quiver & Kodaira matrix & Cartan form $Q$ & Monodromy $\mu$ & $(d_1,d_2,d_3)$ \\ \hline
 $A_1^k$ ($k\geq 0$) & $I_{k}$ & $A_{k-1}$ & 	 $\left(\begin{matrix}
		1 &-k \\
	 0 &1
		\end{matrix}\right)$				 		& $(k,0,0)$ \\ \hline
 $A_2$ & $II$ & $A_0$ & 		$\left(\begin{matrix}
		0 &-1 \\
	 1 &1
		\end{matrix}\right)$		& $(1,1,0)$	 \\ \hline
 $A_3$ & $III$ &$A_1$ &	$\left(\begin{matrix}
		0 &-1 \\
	 1 &0
		\end{matrix}\right)$	 	 & $(2,1,0)$	 \\ \hline
 $D_4$ & $IV$ &$A_2$ & 	$\left(\begin{matrix}
		-1 &-1 \\
	 1 &0
		\end{matrix}\right)$			& $(3,1,0)$			 \\ \hline
 \end{tabular}
 \caption{Cases where no lines wrap.}\label{tab:no wrap}
\end{table}

\begin{rmk}\label{NeedFrozen}
Without frozen vectors, the $I_k$ cases, $k\geq 0$, are actually of rank $0$. Thus, although we tend to ignore frozen vectors, they are necessary for constructing these examples. They are also necessary for many other examples~-- this was reflected in Construction \ref{LooijengaConstr} when we required that the vectors $u_1,\dots,u_m$ generate $N_{\?{Y}}$.
\end{rmk}

\begin{prop}[some lines wrap and some do not]\label{SomeWrap} The following are equivalent:
\begin{enumerate}\itemsep=0pt
\item[$1.$] Some {lines} in $U^{\trop}$ wrap, while others do not.
\item[$2.$] Some $($but not all$)$ sheets of the {developing map} are convex.
\item[$3.$] Cluster varieties corresponding to~$U$ are {acyclic} but not of {finite type}.
\item[$4.$] The {monodromy} satisfies $\Tr(\mu) \leq -2$, and if there is equality, then $\mu$ is conjugate to $
	\left(\begin{smallmatrix}
		-1 &a \\
		0 &-1
		\end{smallmatrix}\right)$ for some $a<0$.
\end{enumerate}
\end{prop}
\begin{proof}(1)$\Leftrightarrow$(2) is easy, and (1)$\Leftrightarrow$(3) follows immediately from Theorems~\ref{not-all-wrap} and~\ref{no-wrap}. We see the equivalence of (1) with (4) because Theorems \ref{NegD}, \ref{NegSD}, \ref{all-wrap}, and \ref{no-wrap} have shown that all other possibilities for $\mu$ are equivalent to other cases.
\end{proof}

\section{Cluster modular groups}\label{CMG}

Recall the cluster modular group $\Gamma$ from Section~\ref{CMGdef}. We seek to explicitly describe the action of $\Gamma$ on $U^{\trop}$ in every positive rank $2$ case. However, keeping track of frozen variables will overly complicate matters and will obscure certain meaningful symmetries. We therefore define a new group $\Gamma'$ for which we drop the requirement that frozen vectors are permuted by $\Gamma$ (we allow frozen vectors to be mapped anywhere). This may introduce more automorphisms than one wishes to consider (i.e., automorphisms which only affect the frozen parts), so we mod out the subgroup which acts trivially on both $U^{\trop}$ and on the set of non-frozen vectors. $\Gamma$ can be recovered by taking the subgroup of $\Gamma'$ which is the stabilizer of the set of frozen vectors (roughly meaning that the corresponding cluster transformations extend over certain partial compactifications).

\subsection[The action on $U^{\trop}$]{The action on $\boldsymbol{U^{\trop}}$}

The action of $\Gamma'$ on rank $2$ $\s{X}$ induces an action on the tropicalization $U^{\trop}$ of the Looijenga pairs corresponding to fibers of $\s{X}$ (although generic fibers may be permuted by $\Gamma'$, their embedding in $\s{X}$ induces a canonical identification of their tropicalizations). Since $U^{\trop}$ as integral linear manifold depends only on the deformation type of $U$ (which is preserved by $\Gamma'$), we see that any $h\in \Gamma'$ must respect the integral linear structure of $U^{\trop}$. Also, recall that elements of $\Gamma'$ (unlike those of $\wh{\Gamma}$) must preserve the form $\langle \cdot,\cdot\rangle$, hence preserve the orientation of $U^{\trop}$. Thus:
\begin{lem}\label{oil}
The action of $\Gamma'$ on $U^{\trop}$ is oriented integral linear.
\end{lem}

Let $\Aut\big(U^{\trop}\big)$ be the group of all orientation preserving integral linear automoprhisms of~$U^{\trop}$. Consider the action $r\colon \Gamma' \rar \Aut\big(U^{\trop}\big)$. What we plan to describe is the image
\begin{gather*}
 G:=r(\Gamma')\subseteq \Aut\big(U^{\trop}\big).
\end{gather*}
As explained in the following conjecture, we expect that $G$ contains most of the interesting information about $\Gamma'$.
\begin{conj}Elements of the kernel of $r$ can be represented by seed transformations whose only seed isomorphisms are ones such that if $e_i\mapsto e_j$, then $v_i=v_j$. In particular, if~$S$ is totally coprime $($which by Proposition~{\rm \ref{CoprimeEquiv}} is always achievable through a sequence of fiberwise-equivalences and mutation equivalences$)$, then $\Gamma'=G$.
\end{conj}

\begin{conj}\label{AutU-Conj}$G=\Aut\big(U^{\trop}\big)$ for all rank $2$ cases.
\end{conj}
Recall $\nu_{\pm}\in\Aut\big(U^{\trop}\big)$ from Section~\ref{LinearAuts}. We will see that at least these elements are always in $G$. Furthermore, from our descriptions of $G$ below, one can explicitly check Conjecture \ref{AutU-Conj} holds for the all-lines-wrap and no-lines-wrap cases~-- i.e., for the cases where $\mu^{-1}$ is one of Kodaira's monodromies.

We now note that when considering $U^{\trop}$ with its canonical integral linear structure, mutating with respect to a seed vector $e_i$ for some seed $S$ does not change the positions of any of the $v_j$'s in $U^{\trop}$ except for $v_i$. This is because the centers of the blowups corresponding to the $e_j$'s, $j\neq i$, are preserved by mutation, and the divisor containing the center is the one corresponding to $v_j$. Thus, we only have to worry about what happens to $v_i$. This vector is negated with respect to the vector space structure $U^{\trop}_S$. We now interpret what this means in different cases.

As in Section~\ref{mutmono}, we use the notation $\mu_{i,S}$ to indicate that we are mutating a seed $S$ with respect to a vector $e_i$. We let $S_{i_1,\dots,i_k}$ denote the seed obtained from $S$ by mutating with respect to the seed vectors with indices $i_1$, then $i_2$, and so on up through $i_k$.

The next two subsections will describe $G$ explicitly in all positive rank $2$ cases, in particular proving that $G$ in these cases is as in Table \ref{ClusterModularTable}.

\begin{table}[t] \centering
 \begin{tabular}{ | l | l |}
 \hline
 Classification & G 	 \\ \hline
 $I_0$ (toric) & $\SL_2(\bb{Z})$ \\ \hline
 $I_k$ ($k>0$ odd) & $\bb{Z}$ \\ \hline
 $I_k$ ($k>0$ even) & $\bb{Z}\oplus \bb{Z}/2\bb{Z}$ \\ \hline
 $II$ & $\bb{Z}/5\bb{Z}$ \\ \hline
 $III$ & $\bb{Z}/3\bb{Z}$ \\ \hline
 $IV$ & $\bb{Z}/4\bb{Z}$ \\ \hline
 some lines wrap. & $\bb{Z}$ \\ \hline
 $I^*_0$ & $\PSL_2(\bb{Z})$ \\ \hline
 $I^*_k$ ($k>0$) & $\bb{Z}$ \\ \hline
 $II^*$ & $\{\Id\}$ \\ \hline
 $III^*$ & $\{\Id\}$ \\ \hline
 $IV^*$ & $\bb{Z}/2\bb{Z}$ \\ \hline
 \end{tabular}\caption{The isomorphism class of the image $G$ of $\Gamma'\rar \Aut\big(U^{\trop}\big)$ for the positive rank $2$ cases.
 If there are no frozen vectors, then $\Gamma=\Gamma'$.}\label{ClusterModularTable}
 \end{table}

\subsection{When at least some lines do not wrap}\label{Auto-no-wrap}
We first describe $G$ in the cases where at least some lines do not wrap. In the toric case we of course have $G=\Gamma'=\SL_2(\bb{Z})$. Note that in these toric cases, $\nu^{\pm}=-\Id\in \SL_2(\bb{Z})$. To understand $G$ in the other cases with at least some lines not wrapping, we begin by finding elements of $\Gamma'$ in these cases which map to $\nu^{\pm}\in \Aut\big(U^{\trop}\big)$.

\subsubsection[$\nu^{\pm}$ with at least some lines not wrapping:]{$\boldsymbol{\nu^{\pm}}$ with at least some lines not wrapping:} We saw in Lemma~\ref{model} that if a~line $L$ does not wrap, then (ignoring frozen vectors) there is a unique seed $S$ for which each $v_i$ is contained in $\sigma_L\setminus \rho$, where $\sigma_L$ is the cone bounded by~$L$ and~$\rho$ is either boundary ray of this cone. Assume the $v_i$'s are arranged in counterclockwise order $v_1,\dots,v_s$.

Note that any line in $U^{\trop}$ which does not intersect any $\rho_{v_i}$ is also a straight line in $U^{\trop}_S$. In particular, this holds for lines of the form $L_{v_1}^{>0}$ and $L_{v_s}^{<0}$. Thus, $\mu_{e_1}^{\s{X}}$ has the effect of applying~$\nu_+$ to~$v_1$, while $\mu_{e_s}^{\s{X}}$ has the effect of applying~$\nu_-$ to~$v_s$.

Now note that $v_2,\dots,v_s,v_1':=\nu_+(v_1)$ are all contained in $\sigma_{L_{v_1}^{>0}}\setminus \rho_{v_1}$, so we can repeat the process, mutating~$v_2$, then~$v_3$, and so on. Alternatively, we could have done the reverse, mutating $v_s$ first, then $v_{s-1}$, and so on. Since $\nu_{\pm}$ are integral linear automorphisms of~$U^{\trop}$ by Lemma~\ref{nu}, we see that
\begin{gather}\label{m-}
 m_-:=\nu_{-} \circ \mu_{s,S_{1,2,\dots,s-1}}\circ \cdots \circ \mu_{1,S}
\end{gather} is an element of $\Gamma'$, with inverse $m_+:=\nu_+ \circ \mu_{1,S_{s,s-1,\dots,2}}\circ\cdots\circ \mu_{s,S}$. We note that applying $r$ gives
\begin{gather*}
 r(m_{\pm}) = \nu_{\pm}\in G
\end{gather*}
whenever at least some lines do not wrap.

\begin{rmk}\label{subseq}Suppose $S'$ is another seed, mutation-equivalent to $S$ and isomorphic to $S$, thus inducing a seed auto-transformation $g$ of $S$. Then $r(g)(\sigma_{L}\setminus \rho)$ must be a cone $\sigma_{L'}\setminus \rho'\subset U^{\trop}$ (determined by a line $L'$ in $U_{S'}^{\trop}$) which contains the directions $v_i'$ associated to $S'$. Conversely, $S'$ is determined by $\sigma_{L'}\setminus \rho'$. Each such $L'$ can be obtained by rotating the original line $L$, and if we rotate the boundary of $\sigma_{L'}$ past some $v_k$, then the acyclic quiver whose associated $v_i$'s are contained in $\sigma_{L'}\setminus \rho'$ will change via the mutation $\mu_k$. Thus, the seed auto-transformations, modulo those which fix the non-frozen seed vectors (as appear in the $I_k$ cases), are all generated by subsequences of the mutations defining $m_-$ and $m_+$.
\end{rmk}

\subsubsection[The $I_k$ cases $(k\geq 1)$]{The $\boldsymbol{I_k}$ cases $\boldsymbol{(k\geq 1)}$}
For these cases, there are non-trivial cluster auto-transformations which fix the non-frozen seed vectors. If we identify $U^{\trop}_S$ with $\bb{R}^2$, with $v_1$ being identified with $(1,0)$, then these cluster auto-transformations fixing $v_1$ are just the stabilizer of $(1,0)$ in $\SL_2(\bb{Z})$, i.e., the infinite cyclic group generated by
\begin{align*}
a:=\left(\begin{matrix} 1 & 1 \\ 0 & 1 \end{matrix} \right)=\mu^{-1/k}.
\end{align*}
The corresponding seed auto-transformation is the one acting by $a$ on the lattice $U^{\trop}_S$ and fixing all the other seed data. The elements $a$ and $\nu_+$ together are sufficient to generate $G$. Viewing~$\nu^+$ as rotation $180^{\circ}$ clockwise around the origin in $\wt{U}_0^{\trop}$, we see that $\nu_+^2 = \mu^{-1}=a^k$. We also check that $\nu_+^{-1}\circ a \circ \nu_+=a$ (it suffices to evaluate on $(0,-1)$ since we know $a$ acts trivially on $v_1=(1,0)$), so $a$ and $\nu_+$ commute. With these relations, we can describe $G$ explicitly. If~$k$ is odd, then $G\cong \bb{Z}$ with generator $\nu_-\circ a^{(k+1)/2}$. If $k$ is even, then $G\cong \bb{Z}\oplus \bb{Z}/2\bb{Z}$, with the $\bb{Z}$-summand generated by $a$ and the $\bb{Z}/2\bb{Z}$-summand generated by $\nu_-\circ a^{k/2}$.

\subsubsection[The $II$, $III$, and $IV$ cases]{The $\boldsymbol{II}$, $\boldsymbol{III}$, and $\boldsymbol{IV}$ cases} For the $II$ case (i.e., the $A_2$-quiver), a single mutation already produces a seed isomorphic to the initial seed, yielding an element $g\in \Gamma'$ whose square is either $m_+$ or $m_-$, depending on which mutation was performed. This element $g$ generates $\Gamma'$ and satisfies $g^5=\id$, so $\Gamma'\cong \bb{Z}/5\bb{Z}$ in this case (the five elements corresponding to the five chambers of the cluster complex). Of course, $m_+$ and $m_-$ also turn out to generate $\Gamma'$, but a priori it was not obvious that their powers would give the element $g$.

For the $III$ and $IV$ cases, we do not have fractional powers of $m_{\pm}$ in $G$ since no subsequence of the mutations in \eqref{m-} yields an isomorphic seed. To see this in the $III$ case, recall from Table~\ref{tab:no wrap} that this corresponds to an $A_3$ quiver, which for some choice of edge-orientations and some choice of basis corresponds to having $v_1=v_3=(1,0)$ and $v_2=(0,1)$. We see that one must mutate all three of these to get back to a configuration with $2$ $v_i$'s in one direction and the third counter-clockwise of this. Similarly for the $IV$ case with the $D_4$ quiver, this time with~$3$~$v_i$'s in the $(1,0)$ direction and still one in the $(0,1)$ direction. Now using Remark~\ref{subseq}, we see that~$G$ in the $III$ and $IV$ cases is generated by $\nu_+$ (or $\nu_-$), which one explicitly checks (cf.\ Example \ref{DevIIIEx} below) has order $3$ or $4$, respectively, yielding $G\cong \bb{Z}/3\bb{Z}$ for the $III$ case and $G\cong \bb{Z}/4\bb{Z}$ in the~$IV$ case.

\begin{eg}\label{DevIIIEx}By Table~\ref{tab:no wrap}, we can view $(Y,D)$ in the type $III$ case as being obtained by starting with the toric variety $\big(\bb{P}^2,D=D_1+D_2+D_3\big)$ (the divisors $D_i$ being distinct lines), and then blowing up two points on $D_1$ and one point on $D_2$. We compute the developing map by taking $\rho_1^0$ and $\rho_2^0$ to be the rays generated by $(1,0)$ and $(0,1)$ respectively, and then repeatedly applying~\eqref{linear}. Fig.~\ref{DevIII}(a) depicts the first sheet of the developing map, plus the start of the second sheet, the five rays there being generated by $(1,0)$, $(0,1)$, $(-1,0)$, $(1,-1)$, and $(2,-1)$. Fig.~\ref{DevIII}(b) continues with the rest of the second sheet and the start of the third, the three new rays being generated by $(-1,1)$, $(-1,0)$, and $(0,-1)$. We see that the cone bounded by the rays~$\rho_1^2$ and~$\rho_2^2$ (i.e., the first cone in the third sheet) is a $3\pi$-radian counterclockwise rotation of the original cone bound by~$\rho_1^0$ and~$\rho_2^0$. Thus, $(\nu_-)^3$ agrees with the deck transformation~$\mu^{-2}$ raising points up two sheets, and so $\nu_-$ has order $3$ as an automorphism of~$U^{\trop}$.

\begin{figure}[htb]\centering
\begin{tabular}{ c c }
\xy
{\ar (0,0); (10,0)}; (10,3)*{\rho_1^0};
{\ar (0,0); (0,10)}; (3,10)*{\rho_2^0};
{\ar (0,0); (-10,0)}; (-13,-1)*{\rho_3^0};
{\ar (0,0); (10,-10)}; (12,-10)*{\rho_1^1};
{\ar (0,0); (20,-10)}; (22,-10)*{\rho_2^1};
\endxy
 \hspace{.5 in} \quad & \quad
\xy
{\ar (0,0); (-10,0)}; (-10,-3)*{\rho_1^2};
{\ar (0,0); (0,-10)}; (-3,-10)*{\rho_2^2};
{\ar (0,0); (-10,10)}; (-6,11)*{\rho_3^1};
{\ar (0,0); (10,-10)}; (12,-10)*{\rho_1^1};
{\ar (0,0); (20,-10)}; (22,-10)*{\rho_2^1};
\endxy
 \vspace{.05 in} \\
 (a) & (b) \end{tabular}
\caption{Developing map for type $III$.} \label{DevIII}
\end{figure}

 They type $IV$ case is handled similarly, and there one finds that $(\nu_-)^4=\mu^{-3}$, indicating that~$\nu_-$ has order $4$ in this case.
\end{eg}

We note that powers of the cluster transformations for the $II$, $III$, and $IV$ cases give the trivial cluster transformations described in \cite[Proposition 1.8]{FG1}, for their cases $h=3,4$, and $6$, respectively. The $I_2$ case with no frozen vectors (i.e., $A_1\times A_1$) corresponds to~\cite{FG1}'s $h=2$ case (note that this case is not rank~$2$, cf.\ Remark~\ref{NeedFrozen}).

\subsubsection{Some lines wrap, but some do not} We saw when defining $m_-$ in~\eqref{m-} that $\mu_{k,S_{1,2,\dots,k-1}}\circ \cdots \circ \mu_{1,S}$ produces a seed isomorphic to~$S$ when $k=s$ (the number of non-frozen seed vectors). However, as with the $A_2$ case above, it often happens in the some-lines-wrap cases that the above composition of mutations produces a seed isomorphic to $S$ for some smaller $k$ (e.g., $k=1$ in the $A_2$ case). Choosing the minimal such~$k$ yields an element of $\Gamma'$ which we denote $m_-^{s/k}$, and then $G$ is the infinite cyclic group~$\bb{Z}$ generated by $r\big(m_-^{s/k}\big)=:\nu_-^{s/k}$. Here, we use Remark~\ref{subseq} to ensure that this gives all of~$G$.

\subsection{When all lines wrap}\label{Auto-wrap}

\subsubsection[The $I_0^*$ case]{The $\boldsymbol{I_0^*}$ case}\label{I0star} Here, $U^{\trop}=\bb{R}^2/\{\pm \Id\}$, so $\Aut\big(U^{\trop}\big)=\SL_2(\bb{Z})/{\pm \id} = \PSL_2(\bb{Z})$. We will now show that $G$ includes this whole automorphism group.

Take a coprime seed as in the second part of Example~\ref{cubic cluster}. That is, take the seed $S$ with no frozen vectors and with $\langle \cdot,\cdot\rangle$ given by
	$\left(\begin{smallmatrix}
		0 &1 &-1 \\
	 -1 &0 &1 \\
	 1 &-1 &0
		\end{smallmatrix}\right)$,
		and each $d_i=d_i'=2$. We can identify $v_1$, $v_2$, and $v_3$ with $(2,0)$, $(0,2)$, and $(-2,-2)$ in $N_{\?{Y}}$, respectively. Mutating with respect to~$v_3$, we have $v_1'=\mu_{3,S}(v_1)=v_1$, $v_2'=\mu_{3,S}(v_2)=(-4,-2)$, and $v_3'=\mu_{3,S}(v_3)=(2,2)$. Note that there is a vector space isomorphism~$\alpha$ taking the ordered triplet $(v_1',v_3',v_2')$ to the ordered triplet $(v_1,v_2,v_3)$. Thus, $\alpha \circ \mu_{3,S}$ gives an element of~$\Gamma$ which induces an automorphism $\tau \in G$.
		
Note that $\tau$ takes the ordered triplet $(v_1,v_2,v_3)$ to the ordered triplet $(v_1,v_1+v_2, v_2)$ (addition done in $\sigma_{v_1,v_2}$ as defined in Notation~\ref{UtropNotation}). Thus, in the basis used in the previous paragraph, $\tau$ is given by the matrix $\tau=\left(\begin{smallmatrix} 1 & 1 \\ 0 & 1\end{smallmatrix}\right)$. On the other hand, we have another cluster auto-transformation given by the seed automorphism taking $e_1\mapsto e_2$, $e_2\mapsto e_3$, and $e_3\mapsto e_1$. The matrix for this latter transformation is $\left(\begin{smallmatrix} 0 & -1 \\ 1 & 1\end{smallmatrix}\right)=\sigma \tau$, where $\sigma:=\left(\begin{smallmatrix} 0 & -1 \\ 1 & 0\end{smallmatrix}\right)$. It is standard that these matrices $\tau$ and $\sigma$ generate the modular group $\PSL_2(\bb{Z})$. Hence, $G=\Aut\big(U^{\trop}\big)=\PSL_2(\bb{Z})$.
		
Note that since we had no frozen vectors, we have $\Gamma=\Gamma'$. Furthermore, $\ker(r)=0$ for this~$S$, so we also have $\Gamma'=G$. Hence, $\Gamma=\PSL_2(\bb{Z})$, agreeing with~\cite{FG1}'s computation of this~$\Gamma$ in their Lemma~2.32.

\subsubsection[The $I_k^*$ cases $(k\geq 1)$]{The $\boldsymbol{I_k^*}$ cases $\boldsymbol{(k\geq 1)}$}

For the $I_k^*$ cases, take $S$ as in Section~\ref{I0star} above, but with $d_1'=2+k$, $d_2'=d_3'=2$. Then we obtain an element $g$ of $G$ exactly as in the $I_0^*$ case, applying $\mu_{3,S}$ followed by a vector space isomorphism. However, unlike before, we cannot cycle the roles of the three seed vectors, because now $v_1$ is special. So the only elements of $\Gamma'$ come from repeatedly applying $g$ or its inverse (the inverse corresponds to using $\mu_{2,S}$ in place of $\mu_{3,S}$ above). Thus, we find that $\Gamma'=\bb{Z}$. The action on $U^{\trop}$ is given as follows: Identify $U^{\trop}\setminus \rho_{v_1}$ with the upper half plane so that $v_2$ and $v_3$ are identified with $(0,1)$ and $(-1,1)$, respectively. Then $a\in \bb{Z}$ corresponds to the automorphism $g^a\colon (x,y)\mapsto (x+ay,y)$. In particular, we have that $\pm k\in \bb{Z}$ corresponds to $\nu_{\pm}$.

\subsubsection[The $IV^*$, $III^*$, and $II^*$ cases]{The $\boldsymbol{IV^*}$, $\boldsymbol{III^*}$, and $\boldsymbol{II^*}$ cases}

For the $IV^*$, $III^*$, and $II^*$ cases, we take $d_2'=3$, $d_3'=2$, and $d_1'=3,4,$ or $5$, respectively. In the $IV^*$ case, when we apply $\mu_{3,S}$, we can then compose with the seed isomorphism $v_3'\mapsto v_3$, $v_1'\mapsto v_2$, and $v_2'\mapsto v_1$. We claim that this is the only nontrivial element of $G$ in this case, yielding $G\cong \bb{Z}/2\bb{Z}$. One can check that this non-trivial element is in fact $\nu_{+}=\nu_-$. In the $III^*$ and $II^*$ cases, we claim that we do not even have this element, i.e., $G$ is trivial.

To check that there are not additional elements of $G$ in these cases, it suffices to check that there are no other elements $\Aut\big(U^{\trop}\big)$. Recall that elliptic matrices (those whose trace has absolute value less than $2$) are conjugate to rotation matrices, and so they can only commute with other elliptic matrices or $\pm \Id$. Up to conjugation, all the elliptic matrices appear in the~$II$,~$III$, or $IV$ rows of Table~\ref{tab:no wrap} or the~$II^*$,~$III^*$ or $IV^*$ rows in Table~\ref{wraptable}. By inspecting these, one sees that the centralizer $C_{\SL_2(\bb{Z})}(\mu)$ in the $IV^*$ and $II^*$ cases is the order $6$ cyclic group generated by $\left(\begin{smallmatrix} 1 & 1 \\ -1 & 0\end{smallmatrix}\right)$, while in the $III^*$ case $\mu=\left(\begin{smallmatrix} 0 & -1 \\ 1 & 0\end{smallmatrix}\right)$ generates its own (order $4$) centralizer. Viewing the elements of these centralizers as counterclockwise rotation (for some oriented basis) by some angle in $[0,2\pi)$, we see that $\Aut\big(U^{\trop}\big)$ is given by the elements of the centralizer of~$\mu^{-1}$ which correspond to smaller rotations than $\mu^{-1}$. This yields the matrix $\left(\begin{smallmatrix} 0 & -1 \\ 1 & 1\end{smallmatrix}\right)$ as the sole non-trivial element of $\Aut\big(U^{\trop}\big)$ in the $IV^*$ case, and shows that $\Aut\big(U^{\trop}\big)={\Id}$ in the~$III^*$ and~$II^*$ cases, as claimed.

\begin{rmk}\label{ExtendedModular}We note that there is an orientation {\it reversing} automorphism in each of these three cases which, after mutating with respect to~$v_3$ takes $v_i'\mapsto v_i$ for each~$i$. Thus, one can obtain extra, potentially interesting symmetries of the scattering diagram by considering $\wh{\Gamma}$ as in Section~\ref{CMGdef} in place of~$\Gamma$.
\end{rmk}

In the $I_0^*$, $III^*$, and $II^*$ cases, one can check that $\nu_{\pm}$ are trivial since both ends of any line point in the same direction, cf.\ \cite[Fig.~4.2(c,d)]{Man1} for the $III^*$ and $II^*$ cases. Thus, in conjunction with what we have seen in the other cases, we have found that:
\begin{prop}\label{nuAut}$\nu_{\pm}$ are induced by the cluster modular group $\Gamma'$ $($which we do not require to preserve frozen vectors$)$ in all the positive cases.
\end{prop}

\subsection{Strong deformation equivalence}

We see from Theorems~\ref{all-wrap} and~\ref{no-wrap} (and the supporting Tables~\ref{wraptable} and~\ref{tab:no wrap}) that if $\mu^{-1}$ is any of Kodaira's monodromies, then $U^{\trop}$ uniquely determines $U$ up to isomorphism and deformation. In fact, we have something slightly stronger. We say that $U$ and $U'$ corresponding to the same~$U^{\trop}$ are {\it strongly deformation equivalent} if we can deform one to the other while preserving the identifications of their divisorial valuations with $U^{\trop}(\bb{Z})$ (cf.\ Remark~\ref{divisorial}). In other words, $U$ and $U'$ being strongly deformation equivalent means that if we decorate the surfaces using their relationship with~$U^{\trop}$, we can find a deformation of~$U$ which is isomorphic to a deformation of~$U'$ via an isomorphism which acts trivially on~$U^{\trop}$.

\begin{thm}\label{sing}Suppose the monodromy of $U^{\trop}$ is one of Kodaira's monodromies. Then $U^{\trop}$ determines $U$ up to strong deformation equivalence. In other words, for any log Calabi--Yau surface $U$ corresponding to~$U^{\trop}$, any factorization of the singularity of $U^{\trop}$ into focus-focus singularities as in Remark~{\rm \ref{factor}} is induced by a toric model of~$U$.
\end{thm}
\begin{proof}As noted at the start of this subsection, we already saw in Tables \ref{wraptable} and \ref{tab:no wrap} that~$U^{\trop}$ in these cases determines $U$ up to deformation and isomorphism. So if $U$ and $U'$ are two log Calabi--Yau surfaces with tropicalization $U^{\trop}$, then some deformation of $U$ is isomorphic to some deformation of $U'$. It only remains to check that this isomorphism $\iota$ can be chosen in a way which acts trivially on $U^{\trop}$. But we have seen in Sections~\ref{Auto-no-wrap} and~\ref{Auto-wrap} that Conjecture~\ref{AutU-Conj} holds in all cases corresponding to Kodaira's monodromies. So if $\iota$ induces a non-trivial automorphism of $U^{\trop}$, this element must be induced by some $g\in \Gamma'$. Thus, applying $g^{-1}$ to the $\s{X}$-space $\s{X}'$ of which $U'$ is a symplectic leaf yields an isomorphism $g^{-1}\circ \iota$ from $U$ to a fiber of $\s{X}'$, and this fiber is a deformation of~$U'$ since they live in the same $\s{X}$-space. Since $g^{-1}\circ \iota$ induces the identity on~$U^{\trop}$, this proves the claim.
\end{proof}

\begin{rmk}We suggest that the appearance of Kodaira's matrices may have some geometric significance. The symplectic heuristic behind \cite{GHK1}'s mirror construction (see their Section~0.6.1) assumes that $U$ admits a special Lagrangian torus fibration over $U^{\trop}$, or at least over a deformation of $U^{\trop}$ in which the singularity is factored into several singular points. Indeed, \cite{Sym}~shows that there is at least a Lagrangian fibration when the singularity is factored. In the~$I_k$ cases there are explicit formulas for special Lagrangian fibrations~-- see~\cite{GrSlag}, or see~\cite{CU} which begins with a nice brief presentation of this. Further examples of cluster varieties are known to come from moduli of Higgs bundles, in which case the Hitchin fibration is a special Lagrangian fibration. We further hope that $U$ admits a hyperk\"{a}hler structure, and that for some rotation of the complex structure, the SYZ fibration will become an elliptic fibration (this is standard in the Hitchin system cases). The inverse monodromy being a Kodaira matrix then suggests the possibility of compactifying this elliptic fibration with a~fiber at infinity. This elliptic fibration picture also leads us to suspect that \cite{CV}'s results on uniqueness of factorizations of singular fibers of elliptic fibrations is related to our Theorem~\ref{sing}.
\end{rmk}

Theorem~\ref{sing} shows whenever the monodromy of $U^{\trop}$ is one of Kodaira's monodromies, there is an essentially\footnote{The definition of a consistent scattering diagram (cf.~\cite{GHK1}) depends on several choices, including a choice of fan $\Sigma$, a lattice $P^{gp}$, a submonoid $P$, and a multi-valued convex integral piecewise linear function~$\varphi$. When we say there is an essentially canonical consistent scattering diagram, we really mean that these choices uniquely determine the scattering diagram up to the notion of equivalence in~\cite{GPS}. Changing the above data just corresponds to changing the coefficients of the scattering functions.} canonical consistent scattering diagram in $U^{\trop}$~-- namely, the cano\-ni\-cal scattering diagram which~\cite{GHK1} associates to a Looijenga pair whose tropicalization is $U^{\trop}$. Furthermore, I conjecture that the canonical scattering diagrams coming from Looijenga pairs are the only consistent ones in any~$U^{\trop}$ which consist only of outgoing rays.

For other (log) Calabi--Yau manifolds of possibly higher dimension, more complicated singularities may again appear. However, if the monodromies are factored into Kodaira's mono\-dromies, the above discussion suggests that the appropriate consistent scattering diagrams near these singular loci should be canonically determined by the affine structure. I speculate that this may allow one to relax the need for the simplicity assumption in \cite{GS11} and the $A_1$-singularity assumption in \cite{WCS}.

\subsection*{Acknowledgements}

Most of this paper is based on part of the author's thesis, which was written while in graduate school at the University of Texas at Austin. I would like to thank my advisor, Sean Keel, for introducing me to these topics and for all his suggestions, insights, and support. I would also like to thank Yan Zhou for asking insightful questions that helped improve the final draft, and also the anonymous referees for their many very helpful suggestions.

This work was supported in part by the center of excellence grant ``Centre for Quantum Geometry of Moduli Spaces'' from the Danish National Research Foundation (DNRF95), and later by the National Science Foundation RTG Grant DMS-1246989, and finally by the Starter Grant ``Categorified Donaldson--Thomas Theory'' no. 759967 of the European Research Council.

\pdfbookmark[1]{References}{ref}
\LastPageEnding

\end{document}